\documentclass[a4paper,twoside]{article}
\usepackage{a4}

\usepackage{hyperref}     
\hypersetup{
	colorlinks=true,        
	linkcolor=blue,         
	citecolor=blue,
	urlcolor=blue,
}

\usepackage{graphicx}
\usepackage{amssymb}
\usepackage{epstopdf}
\usepackage{amsmath,amsthm}
\numberwithin{equation}{section}
\numberwithin{figure}{section}

\usepackage{upref}
\usepackage[active]{srcltx}
\usepackage{xypic}

\usepackage{amsthm,amsmath,amssymb}

\usepackage{mathrsfs}

\usepackage{graphicx} 
\usepackage{amsmath, bm}

\theoremstyle{plain}
\newtheorem{thm}{Theorem}[section]

\newtheorem{lem}[thm]{Lemma}

\newtheorem{definition}[thm]{Definition}

\newtheorem{remark}[thm]{Remark}

\begin{document}
	\title{Pullback theorem and rigidity for Sobolev mappings on Carnot groups}
	\author{Yihan Cui}
	\maketitle
	\begin{abstract}
		This paper establishes a pullback theorem via mollification to extend the rigidity theory of Sobolev mappings between Carnot groups to the low-integrability regime where the Sobolev exponent $p$ is less than the homogeneous dimension $\nu_1$ of the source group. The core technical achievement is the rigorous analysis of the convergence of mollified approximations $f_\varepsilon$ for a mapping $f \in W^{1,p}$, demonstrating that the pullbacks of left-invariant differential forms converge appropriately. This allows for the recovery of more properties of Pansu differentiable mappings.
		The main results are: (1) A generalization of the rigidity theorem to the range $p<\nu_1$ for $f\in W^{1,p}$. (2) When $p>Q$, such mappings are shown to be locally Hölder continuous with exponent $1-\frac{Q}{p}$, where $p>Q$ and $Q=\max \left\{\text{homogeneous dim of}\ G_i\right\}$. (3) with the stratified structure of contact Sobolev mappings, we generalize the non-embedding theorem to contact Sobolev mappings.
	\end{abstract}
	\tableofcontents
	\section{Introduction}
	
	Carnot groups occupy a privileged position in sub-Riemannian geometry. By Mitchell's tangency theorem they are the infinitesimal models of equiregular sub-Riemannian manifolds \cite{mitchell1985carnot}; they arise as the model boundaries of the rank-one symmetric spaces; and they carry both a one-parameter family of dilations $\delta_{\lambda}$ and a Carnot--Carath\'eodory (sub-Riemannian) distance which is $1$-homogeneous with respect to them. The homogeneous dimension $\nu=\sum_{k}k\dim\mathfrak{g}^{[k]}$ --- which is strictly larger than the topological dimension whenever the group is nonabelian --- is the exponent that governs the scaling of volume and hence the form of the isoperimetric, Sobolev and Sobolev--Poincar\'e inequalities on the group \cite{capogna2007introduction,hajlasz2000sobolev}. It is therefore the exponent against which any Sobolev mapping theory has to be calibrated.
	
	The geometric mapping theory on Carnot groups began with Pansu's proof of the quasi-isometric rigidity of the rank-one symmetric spaces \cite{pansu1989}. Pansu introduced the notion of differentiability which now bears his name: a mapping $f:G_{1}\supset U\to G_{2}$ is Pansu differentiable at $x$ if the mappings $y\mapsto\delta_{\lambda^{-1}}\bigl(f(x)^{-1}f(x\delta_{\lambda}y)\bigr)$ converge, locally uniformly in $y$, to a homogeneous group homomorphism $D_{P}f(x)$. Two features of this notion have shaped the subject ever since. First, $D_{P}f(x)$ is not an arbitrary linear map but a \emph{graded Lie algebra homomorphism}; thus the whole differential is determined by its horizontal part $D_{H}f(x)$, an overdetermined algebraic constraint on a first-order object. Second, that constraint is so restrictive that families of mappings which in $\mathbb{R}^{n}$ are infinite-dimensional may collapse, on a Carnot group, to finite-dimensional families of contact transformations. This is the rigidity phenomenon, first observed for quasiconformal mappings on the Heisenberg group \cite{koranyi1985,koranyi1995} and subsequently pursued for quasiregular and quasiconformal mappings on general Carnot groups \cite{heinonen1997quasiregular,capogna1999regularity,xie2017rigidity}. Over the last two decades it has been developed into a systematic structure theory for quasiconformal, quasisymmetric, and --- more generally --- Sobolev mappings between Carnot groups \cite{kleiner2021pansu,kleiner2021nonrigid,kleiner2021rumin}.
	
	A decisive technical instrument in this development is the \emph{Pullback Theorem}. In Euclidean space the corresponding statement goes back to Reshetnyak \cite{reshetnyak1989}: for a Sobolev mapping $f\in W^{1,p}$ the pointwise pullback $f^{*}\omega$ of a smooth form is distributionally well defined and, provided the integrability exponents dictated by the degrees of the forms involved are respected, it commutes with exterior differentiation, $d(f^{*}\omega)=f^{*}(d\omega)$. Kleiner, M\"uller and Xie established the sub-Riemannian counterpart \cite{kleiner2021pansu}: replacing the classical differential by the Pansu differential, they defined the \emph{Pansu pullback} $f_{P}^{*}\omega$ and proved that it commutes with $d$ for $f\in W^{1,p}_{\mathrm{loc}}$ with $p$ above a threshold dictated by the weights of the forms --- in particular for $p>\nu_{1}$, the homogeneous dimension of the source. Among the consequences is a far-reaching rigidity theorem: if $f\in W^{1,p}_{\mathrm{loc}}$ with $p>\nu_{1}$ and $D_{P}f(x)$ is an isomorphism for almost every $x$, then, after reindexing, the indecomposable nonabelian factors of the source and of the target agree up to isomorphism, and $f$ is locally a product of homeomorphisms between the corresponding factors \cite{kleiner2021pansu}.
	
	The hypothesis $p>\nu_{1}$ is essential to every step of this argument. It is precisely the range in which the Sobolev--Morrey inequality on Carnot groups is available, so that a mapping in $W^{1,p}_{\mathrm{loc}}$ admits a locally H\"older continuous representative of exponent $1-\frac{\nu_{1}}{p}$. Continuity, in turn, is what allows one to approximate $f$ by smooth mappings with pointwise control, to invoke Pansu differentiability almost everywhere, and to pass to the limit in the resulting integral identities.
	
	The purpose of this paper is to investigate the complementary, low-integrability regime in which the Sobolev exponent $p$ is \emph{smaller} than the homogeneous dimension $\nu_{1}$ of the source group, and to ask to what extent the geometric conclusions described above survive there. This is not a routine extension of the existing theory; three genuinely different difficulties have to be overcome.
	
	\emph{(i) The regularity input disappears.} For $p<\nu_{1}$ the Sobolev--Morrey inequality fails: a general $f\in W^{1,p}$ need not possess a continuous representative, need not be Pansu differentiable almost everywhere, and need not admit a Lusin-type approximation by smooth mappings with the pointwise control that the classical arguments require. Consequently even the \emph{statement} of a pullback theorem has to be reformulated. The Pansu pullback can no longer be defined pointwise; it has to be constructed as the distributional limit of the pullbacks $f_{\varepsilon}^{*}\omega$ of the mollifications $f_{\varepsilon}=f*\eta_{\varepsilon}$, and the whole content of the theorem is shifted into the analysis of the convergence $f_{\varepsilon}^{*}\omega\to f_{P}^{*}\omega$ as $\varepsilon\to0$.
	
	\emph{(ii) The stratified structure makes the limit singular.} Mollification on a Carnot group is not adapted to the stratification: it smooths in all directions, horizontal or not. The pullback of a contact form, however, contains a bilinear \emph{area term} (the term $\frac{1}{2}\alpha_{k}^{ij}(A_{i}\,dA_{j}-A_{j}\,dA_{i})$ in the model computations of Section~\ref{sec:pullback}), so that differentiating $f_{\varepsilon}^{*}\omega$ produces products of components whose oscillation is governed by the non-horizontal geometry of the target. The upshot is a systematic loss of integrability which is expressed by the step of the target: on a step-$n$ target the admissible error terms are controlled by the Hardy--Littlewood maximal function of $|D_{H}f|$ only after one has combined the $(1,1)$ and $(1,2)$ Poincar\'e inequalities on the group, and the horizontal differential of the mollification converges to $D_{H}f$ not in $L^{q}$ but merely in $L^{\frac{q}{n+1}}_{\mathrm{loc}}$, while the non-horizontal derivatives of $f_{\varepsilon}^{*}(dx_{k})$ even blow up at the rate $\varepsilon^{-1}$. These weighted estimates are the origin of the weight condition $wt(\omega)+wt(d\eta)\leq\nu_{1}$ which appears in our pullback theorem, and they explain why the low-integrability theory is a genuinely stratified, rather than a merely Euclidean, affair.
	
	\emph{(iii) The weight condition is an obstruction, not an artefact.} It is natural to ask whether the restriction $wt(\omega)+wt(d\eta)\leq\nu_{1}$ can be removed by a more refined argument. It cannot. In Section~\ref{sec:sharpness} we show by contradiction that dropping the condition would force $Z\bigl(\det D_{H}f\bigr)=0$ for every left-invariant field $Z$ belonging to a layer of order at least two, whereas the inversion on the Heisenberg group
	\begin{equation*}
		\mathcal{J}(x,y,t)=\left(\frac{-x}{\|(x,y,t)\|_{\mathbb{H}}^{2}},\frac{-y}{\|(x,y,t)\|_{\mathbb{H}}^{2}},\frac{-t}{\|(x,y,t)\|_{\mathbb{H}}^{4}}\right)
	\end{equation*}
	is smooth and Pansu differentiable away from the origin and has horizontal upper gradient $\|D_{H}\mathcal{J}\|\approx\|p\|_{\mathbb{H}}^{-2}$, which is not invariant under the vertical field $Z$. The condition is therefore sharp, and the range of validity of the theorem is exactly the range permitted by the geometry.
	
	Our way around difficulty (i) is to take mollification itself as the definition, and to extract the needed compactness from maximal function estimates rather than from continuity. To be precise, for $f\in W^{1,q}(\Omega;G_{2})$ we write $f_{\varepsilon}=f*\eta_{\varepsilon}$ with a smooth kernel $\eta_{\varepsilon}$ supported in $B(e,\varepsilon)$ and of unit mass, and we prove that the pullbacks of the horizontal and of the higher-layer left-invariant forms have limits which can be identified, whenever $f$ happens to be Pansu differentiable, with the corresponding components of the Pansu pullback. The identification is possible because every error term produced by the commutator structure of the group is dominated by $\mathcal{M}(|D_{H}f|)$ and $\mathcal{M}(|D_{H}f|^{2})$, so that dominated convergence applies even though no pointwise convergence of the derivatives is available. In this sense the mollification provides a \emph{robust substitute for Pansu differentiability} at low integrability.
	
	\subsection*{Main results}
	
	Our first main result is the convergence theorem for mollified Sobolev mappings. It is stated here for a step-$n$ target; the step-$2$ case is Theorem~\ref{thm-approx} in the text and requires only $q\geq2$.
	
	\begin{thm}\label{thm-intro-approx}
		Let $G_{2}$ be a step-$n$ Carnot group of topological dimension $\nu_{2}$, let $\Omega\subset G_{1}$ be open, and let $f\in W^{1,q}(\Omega;G_{2})$ with $q>n+1$. If $B(x_{o},2r)\subset\Omega$, then for every $0<\varepsilon<r$,
		\begin{equation*}
			D_{H}f_{\varepsilon}\to D_{H}f \quad\text{in } L^{\frac{q}{n+1}}_{\mathrm{loc}}(\Omega;\mathbb{R}^{\nu_{2}}),
		\end{equation*}
		where $f_{\varepsilon}=f*\eta_{\varepsilon}=\int_{G_{1}}f(y)\eta_{\varepsilon}(y^{-1}x)\,dy$ and $\eta_{\varepsilon}$ is smooth, supported in $B(e,\varepsilon)$, and satisfies $\int_{G_{1}}\eta_{\varepsilon}=1$. Moreover the non-horizontal components of the differential are controlled by
		\begin{equation*}
			\|Yf_{\varepsilon}^{*}(\omega_{k})\|_{L^{\frac{q}{2}}(\Omega)}^{\frac{q}{2}}
			\lesssim\|D_{H}f\|_{L^{q}(\Omega)}^{\frac{q}{2}}
			\qquad\text{and}\qquad
			\|Yf_{\varepsilon}^{*}(dx_{k})\|_{L^{q}(\Omega)}^{q}
			\lesssim\frac{1}{\varepsilon}\|D_{H}f\|_{L^{q}(\Omega)}^{q},
		\end{equation*}
		so that the degenerate, higher-layer directions contribute only in the weighted combinations permitted by the stratification.
	\end{thm}
	
	Feeding these estimates into the definition of the Pansu pullback yields the pullback theorem in the low-integrability range (Theorem~\ref{thm-Pull-0}): if $\theta$ is a continuous closed $k$-form on the target and $\eta\in\Omega_{c}^{N-k-1}(U)$ is smooth, then, for $q\geq|wt(\theta)|$ and $wt(\theta)+wt(d\eta)\leq\nu_{1}$,
	\begin{equation*}
		\int_{U}f^{*}(\theta)\wedge d\eta=0,
	\end{equation*}
	which is precisely the statement that the Pansu pullback commutes with exterior differentiation. Note that the integral is defined as $\lim_{\varepsilon\to0}\int_{U}f_{\varepsilon}^{*}(\theta)\wedge d\eta$; no continuity of $f$ is used. As a by-product we obtain that, for a contact Sobolev mapping which is Pansu differentiable almost everywhere, the classical differential has the block form $Df=\begin{bmatrix}D_{V_{1}}f&*\\&D_{V_{2}}f\end{bmatrix}$ with $D_{V_{1}}f=D_{H}f$, and hence $f_{P}^{*}(\theta)\wedge\eta=f^{*}(\theta)\wedge\eta$ almost everywhere in the admissible range of weights.
	
	With the pullback theorem at hand we extend the rigidity theory to the range $p<\nu_{1}$. Here and in what follows, $G=\prod_{i\in I}G_{i}$ and $G'=\prod_{j\in I'}G'_{j}$ are products of nonabelian Carnot groups which do not admit any nontrivial further product decomposition, and we set
	\begin{equation*}
		Q:=\max\bigl\{\text{homogeneous dimension of }G_{i}\bigr\}_{i\in I}.
	\end{equation*}
	
	\begin{thm}\label{thm-intro-rigidity}
		Let $U\subset G$ be open and let $f:U\to G'$ be a Sobolev mapping such that
		\begin{enumerate}
			\item[(1)] $f\in W^{1,p}_{\mathrm{loc}}$ for some $p>Q-1$;
			\item[(2)] $D_{P}f(x)$ exists and is an isomorphism for $\mu$-a.e. $x\in U$.
		\end{enumerate}
		Then, up to reindexing, the collections $\{G_{i}\}_{i\in I}$ and $\{G'_{j}\}_{j\in I'}$ are the same up to isomorphism, and $f$ is locally a product of homeomorphisms: after shrinking $U$ if necessary there is a bijection $\sigma:I\to I'$ such that $\pi_{\sigma(i)}\circ f:U\to G'_{\sigma(i)}$ factors through the projection $\pi_{i}:U\to G_{i}$. In particular, if $U=\prod_{i}U_{i}$ with each $U_{i}$ connected, then $f$ is a product.
	\end{thm}
	
	The exponent $Q$ is, in general, much smaller than the homogeneous dimension $\nu$ of $G$, since the latter is the sum of the homogeneous dimensions of the factors. Theorem~\ref{thm-intro-rigidity} therefore covers a genuinely new range of exponents, and the same remark applies to the regularity which we obtain next.
	
	\begin{thm}\label{thm-intro-holder}
		Under the hypotheses of Theorem~\ref{thm-intro-rigidity}, and assuming in addition $p>Q$, the mapping $f$ is locally $\bigl(1-\frac{Q}{p}\bigr)$-H\"older continuous.
	\end{thm}
	
	This should be compared with the Sobolev--Morrey exponent $1-\frac{\nu}{p}$: the latter is meaningful only for $p>\nu$, whereas Theorem~\ref{thm-intro-holder} already gives continuity for $Q<p\leq\nu$, and it does so with the exponent dictated by the \emph{largest factor} rather than by the whole group. The proof uses the product structure furnished by Theorem~\ref{thm-intro-rigidity} to reduce the estimate to the individual factors, where $|D_{H}f_{i}(x_{i})|\leq|D_{H}f(x)|$. An analogous statement holds for product quotients $G=\tilde{G}/\exp(K)$ (Theorem~\ref{product-quo}): for $p>Q-1$ the Pansu differential permutes the subalgebras of the factors and cosets are locally mapped to cosets, and for $p>Q$ the mapping is locally $\bigl(1-\frac{Q}{p}\bigr)$-H\"older continuous. Here the argument also uses the contact lift of $f$ to the central extension $\tilde{G}$ and the smooth contact lift theorem of \cite{hakavuori2025smooth}.
	
	Finally, the stratified convergence of Theorem~\ref{thm-intro-approx} has consequences that go beyond rigidity: it constrains the existence of embeddings of Euclidean domains into Carnot groups. The classical non-embedding theorem of Gromov \cite{gromov1996}, in the differential-form formulation recently given by Haj{\l}asz, Mirra and Schikorra \cite{hajlasz2025holder}, rules out topological embeddings which are $\gamma$-H\"older continuous with $\gamma>\frac{1}{2}$. Using the Sobolev property in place of continuity, we are able to lower the requirement on $\gamma$.
	
	\begin{thm}\label{thm-intro-nonembedding}
		Let $G_{2}$ be a step-$n$ Carnot group with Lie algebra $\mathfrak{g}_{2}=\mathfrak{g}_{2}^{[1]}\oplus\cdots\oplus\mathfrak{g}_{2}^{[n]}$, and let $N>\dim\mathfrak{g}_{2}^{[1]}$. If $p\geq\max\{N,n+1\}$ and $\Omega\subset\mathbb{R}^{N}$ is a domain, then there is no topological embedding $f\in C^{0}\cap W^{1,p}(\Omega;G_{2})$. If the continuity assumption is dropped, then for every $f\in W^{1,p}(\mathbb{B}^{N};G_{2})$ with $p\geq\max\{N,n+1\}$ one has
		\begin{equation*}
			\mathcal{H}^{N-1}\bigl(f(\mathbb{S}^{N-1}(r))\bigr)=0 \qquad\text{for every } r>0.
		\end{equation*}
	\end{thm}
	
	For $p>N$ the assertion follows from the Sobolev embedding theorem together with the algebraic rank obstruction, exactly as in the H\"older category; the point of Theorem~\ref{thm-intro-nonembedding} is the range $n+1<p\leq N$, where $f$ need not be Pansu differentiable almost everywhere and the only surviving information is the stratified structure of the mollified pullbacks. We also record, in the closing part of Section~\ref{sec:nonembedding}, sharp Euclidean Hausdorff measure estimates for the image of a domain under a mapping in $C^{0,\gamma}$ into a step-three Carnot group, in the spirit of the comparison of Euclidean and Heisenberg Hausdorff measures of \cite{balogh2003comparison}.
	
	\subsection*{Organization of the paper}
	
	Section~\ref{sec:prelim} collects the preliminaries on Carnot groups: stratified Lie algebras, homogeneous dimension, dilations, the Carnot--Carath\'eodory distance and the quasi-metric $d_{K}$, the weight of a differential form, and the equivalent definitions of quasiconformality. Section~\ref{sec:pullback} is the technical core of the paper; there we prove the convergence theorems for mollified Sobolev mappings, first for step-two targets (Theorem~\ref{thm-approx}) and then, with the additional scaling difficulties caused by the higher layers, for step-$n$ targets (Theorem~\ref{thm-approx-higher}), and we deduce from them the pullback theorems (Theorems~\ref{Pullback}, \ref{Pullback2} and \ref{thm-Pull-0}). Section~\ref{sec:rigidity} contains the applications to rigidity: the product rigidity theorem (Theorem~\ref{thm-Rigidity}), the H\"older continuity result (Theorem~\ref{Holder}), and the corresponding statements for product quotients (Theorem~\ref{product-quo}). Section~\ref{sec:sharpness} is devoted to the sharpness of the weight condition $wt(\omega)+wt(d\eta)\leq\nu_{1}$, proved by contradiction by means of the Heisenberg inversion. Section~\ref{sec:nonembedding} deals with the non-embedding theorem for contact Sobolev mappings, with the vanishing of the $(N-1)$-dimensional Hausdorff measure of the image of hypersurfaces when continuity is not assumed, and with Hausdorff dimension estimates for H\"older continuous mappings.
	
	\noindent\textbf{Convention.} Unless stated otherwise, $X$ denotes a left-invariant horizontal vector field on the source group, $Y$ a left-invariant field belonging to the second layer and $Z$ a left-invariant field belonging to the third layer; $\nu_{1}$ and $N$ are the homogeneous and the topological dimension of the source, $\nu_{2}$ the topological dimension of the target, and $\mathcal{M}$ the Hardy--Littlewood maximal operator. All mappings are assumed to be measurable and defined on open sets, and all statements are local.
	
	\section{Preliminary}\label{sec:prelim}
	We recall the basic properties of Carnot groups. 
	A Lie algebra  $\mathfrak{g}$  is stratified if it has a decomposition  $\mathfrak{g}=\mathfrak{g}^{[1]} \oplus \cdots \oplus \mathfrak{g}^{[s]}$  with  $\left[\mathfrak{g}^{[1]}, \mathfrak{g}^{[k]}\right]=\mathfrak{g}^{[k+1]}$  for all $ k \geq 1$ , where we denote  $\mathfrak{g}^{[k]}=\{0\}$  for  $k \geq s+1 $. A (sub-Riemannian) Carnot group is a simply connected nilpotent Lie group $ G$  with a stratified Lie algebra  $\mathfrak{g}$  equipped with an inner product on the horizontal layer  $\mathfrak{g}^{[1]}$ . For convenience, We will need an inner product not only on the horizontal layer  $\mathfrak{g}^{[1]}$ , but instead on the whole Lie algebra  $\mathfrak{g}$ . We always equip the Lie algebra of a Carnot group with an inner product for which the layers  $\mathfrak{g}^{[1]} \oplus \cdots \oplus \mathfrak{g}^{[s]}$  are pairwise orthogonal. The homogeneous dimension of $ G$  is  $Q_{G}:=\sum_{k=1}^{s} k \operatorname{dim}\left(\mathfrak{g}^{[k]}\right) $. The (horizontal) rank of  $G$  is the dimension of the horizontal layer  $\mathfrak{g}^{[1]} $.
	
	Given $ g \in G$ , we denote by  $L_{g}: G \rightarrow G$  the left-translation  $L_{g}(h)= g h$ . We identify  $T_{e} G$  with  $\mathfrak{g}$ . An absolutely continuous curve  $\gamma:[0,1] \rightarrow G$  is horizontal if for almost every  $t \in[0,1] $ its left-trivialized derivative  $\left(L_{\gamma(t)}^{-1}\right) * \dot{\gamma}(t)$  is contained in the horizontal layer $ \mathfrak{g}^{[1]} $. The length of the horizontal curve  $\gamma$  is
	\begin{equation*}
		\ell(\gamma)=\int_{0}^{1}\left\|\left(L_{\gamma(t)}^{-1}\right)_{*} \dot{\gamma}(t)\right\| d t
	\end{equation*}
	where the norm is the one induced by the inner product on  $\mathfrak{g}^{[1]}$ . The sub-Riemannian distance between two points  $g, h \in G $ is
	\begin{equation*}
		d(g, h)=\inf \{\ell(\gamma) \mid \gamma:[0,1] \rightarrow G \text { horizontal, } \gamma(0)=g, \gamma(1)=h\} .
	\end{equation*}
	
	By construction, this distance is left-invariant.
	Carnot groups also admit a one-parameter group of automorphisms, known as the dilations  $\delta_{\lambda}: G \rightarrow G, \lambda>0 $. The associated Lie algebra automorphisms, also denoted  $\delta_{\lambda}: \mathfrak{g} \rightarrow \mathfrak{g} $, are the linear maps defined on each layer of the stratification as
	\begin{equation*}
		\delta_{\lambda}(X)=\lambda^{k} X, \quad X \in \mathfrak{g}^{[k]}, \quad 1 \leq k \leq s .
	\end{equation*}
	The sub-Riemannian distance is 1-homogeneous with respect to these dilations, i.e.,  $d\left(\delta_{\lambda} g, \delta_{\lambda} h\right)=\lambda d(g, h) $.
	
	\begin{definition}
		If a differential form $\alpha$ does not vanish a.e. in $U$ we define
		\begin{equation*} 
			\mathrm{weight}(\alpha) := \underset{x \in U: \alpha(x) \neq 0}{\mathrm{ess\,sup}}\ \mathrm{weight}(\alpha(x)).
		\end{equation*}
		If $E \subset U$ is measurable and the restriction of $\alpha$ to $E$ does not vanish a.e., then we define
		\begin{equation*} 
			\mathrm{weight}(\alpha|_{E}) := \underset{x \in E: \alpha(x) \neq 0}{\mathrm{ess\,sup}}\ \mathrm{weight}(\alpha(x)).
		\end{equation*}
	\end{definition}

	\section{Pullback theorem for Sobolev mappings on Carnot groups}\label{sec:pullback}
	For convenience, we just take step three Carnot group as an example to showcase the result, as for higher step Carnot groups the conclusion holds as well. 
	
	Let  $G $ be a step three Carnot group. Let  $d_{1}, \ldots, d_{r}$  be a basis for  $\mathfrak{g}_{1}, e_{1}, \ldots, e_{s}$  a basis for  $\mathfrak{g}_{2} $, and  $f_{1}, \ldots, f_{t}$  a basis for  $\mathfrak{g}_{3}$ . We write
	\begin{equation*}
		\begin{array}{l}
			{\left[d_{i}, d_{j}\right]=\sum_{k=1}^{s} \alpha_{k}^{i j} e_{k}} \\
			{\left[d_{i}, e_{k}\right]=\sum_{m=1}^{t} \beta_{m}^{i k} f_{m}}
		\end{array}
	\end{equation*}
	with all other bracket relations trivial.
	As in the step two case, we can identify  $G$  with  $\mathbb{R}^{r+s+t}$  equipped with the following operation via coordinates of the first kind:
	\begin{equation*}
		\left(A_{i}, B_{k}, C_{m}\right) \star\left(a_{i}, b_{k}, c_{m}\right)=\left(\mathcal{A}_{i}, \mathcal{B}_{k}, \mathcal{C}_{m}\right)
	\end{equation*}
	where
	\begin{equation*}
		\begin{array}{l}
			\mathcal{A}_{i}=A_{i}+a_{i} \\
			\mathcal{B}_{k}=B_{k}+b_{k}+\frac{1}{2} \sum_{i<j} \alpha_{k}^{i j}\left(A_{i} a_{j}-a_{i} A_{j}\right) \\
			\mathcal{C}_{m}=C_{m}+c_{m}+\frac{1}{2} \sum_{i, j} \beta_{m}^{i j}\left(A_{i} b_{j}-B_{j} a_{i}\right)+\frac{1}{12} \sum_{l, k} \sum_{i<j}\left(A_{l}-a_{l}\right) \alpha_{k}^{i j}\left(A_{i} a_{j}-a_{i} A_{j}\right) \beta_{m}^{l k}
		\end{array}
	\end{equation*}
	$\operatorname{Observe}\left(A_{i}, B_{k}, C_{m}\right)^{-1}=\left(-A_{i},-B_{k},-C_{m}\right)$  and
	
	\begin{equation*}
		\left(A_{i}, B_{k}, C_{m}\right)^{-1} \star\left(a_{i}, b_{k}, c_{m}\right)=\left(\tilde{\mathcal{A}_{i}}, \tilde{\mathcal{B}_{k}}, \tilde{\mathcal{C}_{m}}\right),
	\end{equation*}
	where
	\begin{equation}
		\begin{array}{l}
			\tilde{\mathcal{A}}_{i}=a_{i}-A_{i} \\
			\tilde{\mathcal{B}_{k}}=b_{k}-B_{k}-\frac{1}{2} \sum_{i<j} \alpha_{k}^{i j}\left(A_{i} a_{j}-a_{i} A_{j}\right) \\
			\tilde{\mathcal{C}_{m}}=c_{m}-C_{m}-\frac{1}{2} \sum_{i, j} \beta_{m}^{i j}\left(A_{i} b_{j}-B_{j} a_{i}\right)+\frac{1}{12} \sum_{l, k} \sum_{i<j}\left(A_{l}+a_{l}\right) \alpha_{k}^{i j}\left(A_{i} a_{j}-a_{i} A_{j}\right) \beta_{m}^{l k}
		\end{array}
	\end{equation}\label{metric-coord}
	Left-translating the canonical basis at the origin, we obtain the left-invariant vector fields
	\begin{equation*}
		\begin{aligned}
			X^{i}= & \frac{\partial}{\partial A_{i}}-\frac{1}{2}\sum_{k=1}^{s} \sum_{i<j} \alpha_{k}^{i j} A_{j} \frac{\partial}{\partial B_{k}} \\
			& +\sum_{m=1}^{t}\left[-\frac{1}{2} \sum_{j=1}^{s} B_{j} \beta_{m}^{i j}+\frac{1}{12} \sum_{l=1}^{r} \sum_{k=1}^{s} A_{l}\left(\sum_{i<j} \alpha_{k}^{i j} A_{j}\right) \beta_{m}^{l k}\right] \frac{\partial}{\partial C_{m}} \\
			Y^{k}= & \frac{\partial}{\partial B_{k}}+\sum_{m=1}^{t}\left(\frac{1}{2} \sum_{i=1}^{r} \beta_{m}^{i k} A_{i}\right) \frac{\partial}{\partial C_{m}} \\
			Z^{m}= & \frac{\partial}{\partial C_{m}}
		\end{aligned}
	\end{equation*}
	It is clear that  $\left\{X^{i}\right\}_{i} \cup\left\{Y^{k}\right\}_{k} \cup\left\{Z^{m}\right\}_{m}$  forms a basis for  $\operatorname{Lie}\left(\mathbb{R}^{n}, \star\right) $. Moreover, we have the expected step three stratification of $ \operatorname{Lie}\left(\mathbb{R}^{n}, \star\right) $ :
	\begin{equation*}
		\operatorname{Lie}\left(\mathbb{R}^{r+s+t}, \star\right)=\left\langle X^{i}\right\rangle_{1 \leq i \leq r} \oplus\left\langle Y^{k}\right\rangle_{1 \leq k \leq s} \oplus\left\langle Z^{m}\right\rangle_{1 \leq m \leq t}
	\end{equation*}
	In fact, one can show using the Jacobi identity that the linear map  $\varphi: \operatorname{Lie}\left(\mathbb{R}^{n}, \star\right) \rightarrow \mathfrak{g}$  induced by
	\begin{equation*}
		X^{i} \mapsto d_{i}, \quad Y^{k} \mapsto e_{k}, \quad Z^{m} \mapsto f_{m}
	\end{equation*}
	is a Lie algebra isomorphism.
	The contact forms are given by
	\begin{equation*}
		\begin{aligned}
			\omega_{1}^{k} & :=d B_{k}-\frac{1}{2}\sum_{i=1}^{r} \sum_{i<j} \alpha_{k}^{i j}\left( A_{i}d A_{j}- A_{j}d A_{i}\right)  \\
			\omega_{2}^{m} & :=d C_{m}+\frac{1}{2}\sum_{i=1}^{r} \sum_{j=1}^{s} B_{j} \beta_{m}^{i j}d A_{i}-\frac{1}{12} \sum_{l=1}^{r} \sum_{k=1}^{s}\sum_{i<j} A_{l} \alpha_{k}^{i j}\left( A_{i}d A_{j}- A_{j}d A_{i}\right) \beta_{m}^{l k} 
		\end{aligned}
	\end{equation*}
	We have
	\begin{equation*}
		H\left(\mathbb{R}^{r+s+t}, \star\right)=\bigcap_{k=1}^{s} \operatorname{ker} \omega_{1}^{k} \cap \bigcap_{m=1}^{t} \operatorname{ker} \omega_{2}^{m}
	\end{equation*}
	so that a tangent vector $ v $ lies in  $H_{p}\left(\mathbb{R}^{r+s+t}, \star\right)$  if and only if  $\left(\omega_{1}^{k}\right)_{p}(v)=\left(\omega_{2}^{m}\right)_{p}(v)=0$  for all $ k$  and  $m $.
	
	Every step $n$ Carnot group has the relation that
	if  $K \subset \mathbb{G}$  is compact, then there is a constant $ C=C(K) \geq 1 $ such that
	\begin{equation}
		C^{-1}|p-q| \leq d_{c c}(p, q) \leq C|p-q|^{1 / n} \quad 
	\end{equation}\label{metric-compar}
	for all $ p, q \in K$ .
	
	We can define the quasi-metric 
	\begin{equation}\label{metric-K}
		d_K(p,q)=\sum_{i=1}^{n}\sum_{j=1}^{dim(V_i)}|\pi_{i}^j(q^{-1}p)|^{\frac{1}{i}}
	\end{equation}
	where $\pi_{i}^j$ is the projection to the i-th layer and the j-th coordinate.
	\begin{equation*}
		\begin{aligned}
			&\pi_{1}^i(q^{-1}p)=a_{i}-A_{i}\\
			&\pi_{2}^k(q^{-1}p):=\varphi_{1}^k(p, q)=b_{k}-B_{k}-\frac{1}{2} \sum_{i<j} \alpha_{k}^{i j}\left(A_{i} a_{j}-a_{i} A_{j}\right)\\
			&\pi_{3}^m(q^{-1}p):=\varphi_{2}^m(p, q)=c_{m}-C_{m}-\frac{1}{2} \sum_{i, j} \beta_{m}^{i j}\left(A_{i} b_{j}-B_{j} a_{i}\right)+\frac{1}{12} \sum_{l, k} \sum_{i<j}\left(A_{l}+a_{l}\right) \alpha_{k}^{i j}\left(A_{i} a_{j}-a_{i} A_{j}\right) \beta_{m}^{l k}
		\end{aligned}
	\end{equation*}
	
	It is worth to mention that $d_K$ is comparable with C-C metric on Carnot group.
	
	It is known that for Pansu differentiable map $f$  the Pansu differential is a Lie algebra homorophism. Thus, if the vector field $X$ and contact form $\omega$ are not from the same layer, then $X(f_P^*\omega)=0$. For instance, on Heisenberg group $\mathbb{H}^1$
	\begin{equation*}
		\begin{aligned}
			D_Pf
			&=\begin{bmatrix}
				X_1(f_P^*dx_1) & X_1(f_P^*dx_2) & 0\\
				X_2(f_P^*dx_1) & X_2(f_P^*dx_2) & 0\\
				0 & 0 & Z(f_P^*\omega_{1})\\
			\end{bmatrix}
		\end{aligned}
	\end{equation*}
	Although the Pansu differential of Hölder continuous maps do not exist, we can generalize the Pansu differential with the help of mollification.
	
	\begin{thm}\label{thm-approx}
		Suppose that  $f \in W^{1,q}\left(\Omega ; G_2\right)$ , where $ G_2$ is a step-2 Carnot group and $\nu_2$ is the topological  dimension of $G_2$, $\Omega \subset G_1$  is open , $X$ is a left-invariant  vector field belonging to the horizontal layer and $q\geq2$ . If  $B\left(x_{o}, 2 r\right) \subset \Omega $, then  for all  $0<\varepsilon<r$,
		\begin{equation}\label{approx}
			D_Hf_\varepsilon \to D_Hf \  \text{in}\  L_{loc}^\frac{q}{2}(\Omega;\mathbb{R}^{\nu_2}).\\
		\end{equation}
		where the convergence is in the sense of the norm on Euclidean space $\mathbb{R}^{\nu_2}$, $f_{\varepsilon}=f * \eta_{\varepsilon}=\int_{\mathbb{G}_1 } f(y) \eta_{\varepsilon}\left(y^{-1} x\right) d y=\int_{\mathbb{G}_1 } f(xy^{-1}) \eta_{\varepsilon}\left(y\right) d y$  and $\eta_{\varepsilon}$ is a smooth function on Carnot group $G_1$ with a compact supported set $B(e,\varepsilon)$ and
		\begin{equation*}
			\int_{G_1}\eta_{\varepsilon}=1.
		\end{equation*}
		\begin{proof}
			Recall that  $\varphi(p, q)$  was defined in \eqref{metric-K}. Let  $p \in B\left(x_{o}, r\right)$  and let  $\varepsilon \in(0, r)$ . In what follows we will identify  $\left(f_{\varepsilon}^{*} \alpha\right)(p)$  with the vector (of equal length):
			\begin{equation*}
				\begin{aligned}
					&X(f_{\varepsilon}^{*} \omega_{1}^k)(p)\\
					=&X( f_{\varepsilon}^{B_{k}})(p)+\frac{1}{2}\alpha_{k}^{ij} \sum_{i<j}\left(f_{\varepsilon}^{A_{j}}(p) X(f_{\varepsilon}^{A_{i}})(p)-f_{\varepsilon}^{A_{i}}(p) X( f_{\varepsilon}^{A_{j}})(p)\right)\\
					=& 
					\varepsilon^{-\nu_1-1} \int_{B_{\varepsilon}}\left(f^{B_{k}}(pz^{-1})-f_{\varepsilon}^{B_{k}}(p)+\frac{1}{2}\alpha_{k}^{ij} \sum_{i<j}\left(f_{\varepsilon}^{A{j}}(p) f^{A_{i}}(pz^{-1})-f_{\varepsilon}^{A_{i}}(p) f^{A_{j}}(pz^{-1})\right)\right)\\
					&\cdot X(\eta\left(\delta_{\varepsilon^{-1}}(z)\right)) d z \\
					=&\varepsilon^{-\nu_1-1} \int_{B_{\varepsilon}} \varphi_{k}\left(f(pz^{-1}), f_{\varepsilon}(p)\right) X(\eta\left(\delta_{\varepsilon^{-1}}(z)\right)) d z \\
					=&\varepsilon^{-\nu_1-1} \int_{B_{\varepsilon}} \int_{B_{\varepsilon}} \varphi_{k}(f(pz^{-1}), f(pw^{-1})) \eta_{\varepsilon}(w) X(\eta\left(\delta_{\varepsilon^{-1}}(z)\right)) d w d z
				\end{aligned}
			\end{equation*}
			In the second equality we used the fact that $ \int_{B_{\varepsilon}} f_{\varepsilon}^{t_{k}}(p) X( \eta\left(\delta_{\varepsilon^{-1}}(z)\right)) d z=0 $. Easy verification of the last equality is left to the reader. Note that \eqref{metric-K} yields
			\begin{equation*}
				|\varphi_{k}(f(pz^{-1}), f(pw^{-1}))| \leq d_{K}(f(pz^{-1}), f(pw^{-1}))^{2} 
			\end{equation*}
			Therefore,
			\begin{equation*}
				\begin{aligned}
					\left|X(f_{\varepsilon}^{*} \omega_{1}^k)(p)\right| 
					&\leq\varepsilon^{-\nu_1-1} \int_{B_{\varepsilon}} \int_{B_{\varepsilon}} \varphi_{k}(f(pz^{-1}), f(pw^{-1})) \eta_{\varepsilon}(w) X(\eta\left(\delta_{\varepsilon^{-1}}(z)\right)) d w d z\\
					&\leq \varepsilon^{-\nu_1-1} \int_{B_{\varepsilon}} \int_{B_{\varepsilon}} d_{K}(f(pz^{-1}), f(pw^{-1}))^{2} \eta_{\varepsilon}(w) X(\eta\left(\delta_{\varepsilon^{-1}}(z)\right)) d w d z\\
					&\lesssim \varepsilon^{-\nu_1-1} \int_{B_{\varepsilon}} \int_{B_{\varepsilon}} d_{cc}(f(pz^{-1}), f(pw^{-1}))^{2} \eta_{\varepsilon}(w) X(\eta\left(\delta_{\varepsilon^{-1}}(z)\right)) d w d z\\
					&\lesssim \varepsilon^{-2 \nu_1-1} \int_{B_{\varepsilon}} \int_{B_{\varepsilon}} d_{cc}(f(pz^{-1}), f(pw^{-1}))^{2}  d w d z\\
				\end{aligned}
			\end{equation*}
			denote $x=pz^{-1}$ and $y=pw^{-1}$,
			thus $z^{-1},w^{-1}\in B_\varepsilon^{-1}=B_\varepsilon$ and $pz^{-1},pw^{-1}\in p \cdot B_\varepsilon$. Since the Haar measure $dw$ and $dz$ is invariant under left and right transformation,
			\begin{equation*}
				\begin{aligned}
					&\varepsilon^{-2 \nu_1-1} \int_{B_{\varepsilon}} \int_{B_{\varepsilon}} d_{cc}(f(pz^{-1}), f(pw^{-1}))^{2}  d w d z\\
					&=\varepsilon^{-2 \nu_1-1} \int_{ p\cdot B_{\varepsilon}} \int_{p\cdot B_{\varepsilon}} d_{cc}(f(x), f(y))^{2}  d x d y\\
					&\leq \varepsilon^{-2\nu_1-1}\cdot C \cdot \varepsilon^2 \int_{B_{\varepsilon}}  \int_{\lambda B_{\varepsilon}(p)} \left|D_H f\right|^2  d x d y\\
					&= C \cdot \varepsilon^{-\nu_1+1}     \int_{\lambda B_{\varepsilon}(p)} \left|D_H f\right|^2  d x \\
				\end{aligned}
			\end{equation*}
			in the inequality we use the $(1,2)$-Poincaré inequality and there is a constant $\lambda$ such that $p\cdot B_{\varepsilon} \subset \lambda B_\varepsilon(p)$. We define the Hardy-Littlewood maximal function 
			\begin{equation}\label{maximal-g}
				\mathcal{M}(g)(p):= \sup_{r>0} \frac{1}{|B(p,r)|}\int_{B(p,r)} |g(x)|dx.
			\end{equation}   
			Then
			\begin{equation*}
				\begin{aligned}
					\left|X(f_{\varepsilon}^{*} \omega_{1}^k)(p)\right| 
					&\lesssim C \cdot \varepsilon^{-\nu_1+1}     \int_{\lambda B_{\varepsilon}(p)} \left|D_H f\right|^2  d x \\
					&\lesssim \varepsilon \mathcal{M}(\left|D_H f\right|^2)(p) 
				\end{aligned}
			\end{equation*}
			and for compact subset $W$,
			\begin{equation*}
				\begin{aligned}
					\|X(f_{\varepsilon}^{*} \omega_{1}^k)\|_{L^{\frac{q}{2}}(W)} 
					&= \left(\int_{W}\left|X(f_{\varepsilon}^{*} \omega_{1}^k)\right|^\frac{q}{2} \right)^\frac{2}{q} \\
					&\lesssim \left(\int_{W}\left|\varepsilon \mathcal{M}(\left|D_H f\right|^2)\right|^\frac{q}{2} \right)^\frac{2}{q} \\ 
					&= \varepsilon \|\mathcal{M}(\left|D_H f\right|^2) \|_{L^{\frac{q}{2}}(W)}\\
					&\leq C_q \cdot\varepsilon \|\left|D_H f\right|^2 \|_{L^{\frac{q}{2}}(W)}\\
					&=C_q \cdot \varepsilon \|D_H f \|_{L^{q}(W)}^2
				\end{aligned}
			\end{equation*}
			Thus, $X(f_{\varepsilon}^{*} \omega_{1}^k)\rightarrow 0$ in $L^\frac{q}{2}(W)$ for $q>2$. As for $q=2$,
			\begin{equation*}
				\begin{aligned}
					\|X(f_{\varepsilon}^{*} \omega_{1}^k)\|_{L^1(W)}   
					&\leq C \cdot \varepsilon^{-\nu_1+1}    \int_{W} \int_{\lambda B_{\varepsilon}(p)} \left|D_H f\right|^2  d x dp \\
					&= C \cdot \varepsilon^{-\nu_1+1}    \int_{G_1} \int_{G_1} \left|D_H f\right|^2(x) \cdot I_{\lambda B_{\varepsilon}(p)}(x)\cdot I_W(p)  d x dp \\
					&= C \cdot \varepsilon^{-\nu_1+1}    \int_{G_1} \int_{G_1} \left|D_H f\right|^2(x) \cdot I_{\lambda B_{\varepsilon}(x)}(p)\cdot I_W(p)   dp d x\\
					&\lesssim C \cdot \varepsilon     \int_{G_1} \left|D_H f\right|^2(x)  d x\\
					&\to 0.
				\end{aligned}
			\end{equation*}
			
			For horizontal layer,
			\begin{equation*}
				\begin{aligned}
					Xf_\varepsilon^*(dx_j)(p)&=X(f_\varepsilon^{x_j})(p)\\
					&= \varepsilon^{-\nu_1-1}  \int_{B_{\varepsilon}} \left(f^{x_j} (pz^{-1})-f^{x_j} (p)\right) X(\eta\left(\delta_{\varepsilon^{-1}}(z)\right))dz\\
					&\lesssim \varepsilon^{-\nu_1-1}  \int_{B_{\varepsilon}} d_{cc}\left(f (pz^{-1}),f (p)\right) X(\eta\left(\delta_{\varepsilon^{-1}}(z)\right))dz\\
					& \leq \varepsilon^{-\nu_1-1}  \int_{B_{\varepsilon}} d_{cc}\left(f (pz^{-1}),f (p)\right) dz\\
					& \leq \varepsilon^{-\nu_1-1} \cdot C \varepsilon \int_{\lambda B_{\varepsilon}(p)} |D_H f|\\
					& \lesssim \varepsilon^{-\nu_1} \int_{\lambda B_{\varepsilon}(p)} |D_H f| \\
					& = \mathcal{M}(\left|D_H f\right|)(p) 
				\end{aligned}
			\end{equation*}
			where we use the $(1,1)$-Poincaré inequality.
			Thus, 
			\begin{equation*}
				\begin{aligned}
					\|Xf_\varepsilon^*(dx_j)\|_{L^{q}(W)} 
					&\lesssim \|\mathcal{M}(\left|D_H f\right|)\|_{L^{q}(W)} \\
					&\leq C_q \| D_H f \|_{L^{q}(W)}. 
				\end{aligned}
			\end{equation*}
			Additionally, we have
			\begin{equation*}
				\begin{aligned}
					\|Xf_\varepsilon^{x_j}(p)\|& \leq \int_{B_{\varepsilon}} |Xf^{x_j}(pz^{-1})| \eta_{\varepsilon}(z) dz\\
					&= \int_{B_{\varepsilon}} |Xf^{x_j}(pz^{-1})| \eta_{\varepsilon}^{\frac{1}{q}}(z)\cdot \eta_{\varepsilon}^{1-\frac{1}{q}}(z) dz\\
					&\leq \left(\int_{B_{\varepsilon}} \left|Xf^{x_j}(pz^{-1})\right|^q \eta_{\varepsilon}(z) dz\right)^{\frac{1}{q}} \left(\int_{B_{\varepsilon}}\eta_{\varepsilon}(z) dz \right)^{1-\frac{1}{q}}
				\end{aligned}
			\end{equation*}
			then 
			\begin{equation*}
				\begin{aligned}
					\int_V \left|Xf_\varepsilon^{x_j}(x)\right|^q dx &\leq \int_{B_{\varepsilon}} \eta_{\varepsilon}(z) \int_V \left|Xf^{x_j}(xz^{-1})\right|^q dx dz\\
					& \leq \int_W \left|Xf^{x_j}(y)\right|^q dy
				\end{aligned}
			\end{equation*}
			Because $Xf^{x_j}\in L^q(W)$, there is a continuous function $g $ such that $\|Xf^{x_j}-g \|_{L^q(W)}\leq \delta$. Therefore, 
			\begin{equation*}
				\|Xf_\varepsilon^{x_j}-g_\varepsilon \|_{L^q(W)}= \|\left(Xf^{x_j}-g\right)* \eta_{\varepsilon} \|_{L^q(W)} \leq \delta.
			\end{equation*}
			Thus,
			\begin{equation*}
				\begin{aligned}
					\| Xf^{x_j}-Xf_\varepsilon^{x_j} \|_{L^q(V)}&\leq \| Xf^{x_j}-g \|_{L^q(V)}+\| g-g_\varepsilon \|_{L^q(V)} +\| Xf_\varepsilon^{x_j}-g_\varepsilon \|_{L^q(V)}\\
					& \leq 2\delta + \| g-g_\varepsilon \|_{L^q(V)}\\
					& \leq 3\delta
				\end{aligned}
			\end{equation*}
			which means $Xf_\varepsilon^{x_j}\rightarrow Xf^{x_j}$ in $L^q$.

		\end{proof}	
	\end{thm}	
	
	\begin{remark}
		There are some important observations :\\
		(1) The convergence of  $X(f_{\varepsilon}^{*} \omega_{1}^k)\rightarrow 0$ in $L^\frac{q}{2}(W)$ does not depend on the choice of compact support of smoothing kernel, since in the proof we use Hardy-Littlewood maximal function.\\
		(2) In theorem \ref{thm-approx}, $p\geq2$ needs to hold. From the proof, a very natural conjecture is that $X(f_{\varepsilon}^{*} \omega_{1}^k)\nrightarrow 0$ for $p<2$ as long as the variation of $f$ in the directions of vector field in the second layer is comparable to the total variation. 
		
		For convenience, we just show the behavior on Heisenberg group $\mathbb{H}^1$. A simple case is that the image is contained in the following area when near singular point.
		\begin{equation*}
			\mathcal{E}_\delta:=\left\{z>\delta(x^2+y^2)\right\}
		\end{equation*}
		for some $\delta>0$.
		That means locally  
		\begin{equation*}
			\begin{aligned}
				d_{K}(f(pz^{-1}), f(pw^{-1}))^{2}\lesssim \varphi_{k}(f(pz^{-1}), f(pw^{-1})).
			\end{aligned}
		\end{equation*}
		Then
		\begin{equation*}
			\begin{aligned}
				\left|X(f_{\varepsilon}^{*} \omega_{1}^k)(p)\right| 
				&\approx \varepsilon^{-\nu_1-1} \int_{B_{\varepsilon}} \int_{B_{\varepsilon}} \varphi_{k}(f(pz^{-1}), f(pw^{-1})) \eta_{\varepsilon}(w) X(\eta\left(\delta_{\varepsilon^{-1}}(z)\right)) d w d z\\
				&\approx \varepsilon^{-\nu_1-1} \int_{B_{\varepsilon}} \int_{B_{\varepsilon}} d_{K}(f(pz^{-1}), f(pw^{-1}))^{2} \eta_{\varepsilon}(w) X(\eta\left(\delta_{\varepsilon^{-1}}(z)\right)) d w d z\\
				&\approx \varepsilon^{-2 \nu_1-1} \int_{B_{\varepsilon}} \int_{B_{\varepsilon}} d_{cc}(f(pz^{-1}), f(pw^{-1}))^{2}  d w d z\\
			\end{aligned}
		\end{equation*} 
		If we can find a mapping $f$ that its oscillation controls its upper gradient 
		\begin{equation*}
			C \cdot \varepsilon^2  \int_{\lambda B_{\varepsilon}(p)} \left|D_H f\right|^2  d x d y
			\leq  \int_{B_{\varepsilon}\cdot p} d_{cc}(f(x), f(y))^{2}  d x d y
		\end{equation*}
		then
		\begin{equation*}
			\left|X(f_{\varepsilon}^{*} \omega_{1}^k)(p)\right|\approx C \cdot \varepsilon^{-\nu_1+1}     \int_{\lambda B_{\varepsilon}(p)} \left|D_H f\right|^2  d x 
		\end{equation*}
		and
		\begin{equation*}
			\begin{aligned}
				\|X(f_{\varepsilon}^{*} \omega_{1}^k)\|_{L^1(W)}   
				&\approx C \cdot \varepsilon^{-\nu_1+1}    \int_{W} \int_{\lambda B_{\varepsilon}(p)} \left|D_H f\right|^2  d x dp \\
				&= C \cdot \varepsilon^{-\nu_1+1}    \int_{G_1} \int_{G_1} \left|D_H f\right|^2(x) \cdot I_{\lambda B_{\varepsilon}(p)}(x)\cdot I_W(p)  d x dp \\
				&= C \cdot \varepsilon^{-\nu_1+1}    \int_{G_1} \int_{G_1} \left|D_H f\right|^2(x) \cdot I_{\lambda B_{\varepsilon}(x)}(p)\cdot I_W(p)   dp d x\\
				&\approx  C \cdot \lambda^{\nu_1} \varepsilon     \int_{W} \left|D_H f\right|^2(x)  d x\\
			\end{aligned}
		\end{equation*}
		which shows that it fails to converge in case $f\notin W^{1,2}(\mathbb{H}^1;\mathbb{H}^1)$.\\
		(3) $\mathcal{E}_\delta$ to some extent is a good condition for constructing a counterexample. 
		
		A opposite case with respect to domain $\mathcal{E}_\delta$ is that the image of mapping $f$ near singular point is controlled above in the $Z$-direction. Because of (3.20) in Corollary 3.16 from \cite{hajlasz2025holder}, at best, the image of a horizontal curve would increase at a rate proportional to the square in the $Z$-direction, if the the image increase at a rate proportional to the first power in the horizontal directions. That is to say, whether $X(f_{\varepsilon}^{*} \omega_{1}^k)\rightarrow 0$ in $L^\frac{q}{2}(W)$ holds depends on whether we can get a good control for $\varphi_{k}(f(pz^{-1}), f(pw^{-1}))$. \\
		(4)Linear control. If there is a control over $\varphi_{k}(f(pz^{-1}), f(pw^{-1}))$ such that $d_{cc}(f(pz^{-1}), f(pw^{-1}))$ increase nearly linearly. That is to say when $d_{cc}(f(pz^{-1}), f(pw^{-1}))$ large enough, the image of $f$ near singular point is in the domain $\mathcal{L}_\delta$. For convenience, suppose that $f\in W^{1,1}(\mathbb{H}^n; \mathbb{H}^n)$. 
		\begin{equation*}
			\mathcal{L}_\delta:=\left\{z\leq \delta (x_1^2+\dots+x_n^2+y_1^2+\dots+y_n^2)^\frac{1}{2}:=r\right\}.
		\end{equation*}
		Thus,
		\begin{equation*}
			\varphi_{k}(f(pz^{-1}), f(pw^{-1})) \lesssim \sqrt{1+\delta^2}d_{cc}(f(pz^{-1}), f(pw^{-1}))
		\end{equation*}
		and
		\begin{equation*}
			\begin{aligned}
				\left|X(f_{\varepsilon}^{*} \omega_{1}^k)(p)\right| 
				&\leq\varepsilon^{-\nu_1-1} \int_{B_{\varepsilon}} \int_{B_{\varepsilon}} \varphi_{k}(f(pz^{-1}), f(pw^{-1})) \eta_{\varepsilon}(w) X(\eta\left(\delta_{\varepsilon^{-1}}(z)\right)) d w d z\\
				&\lesssim \varepsilon^{-2 \nu_1-1} \int_{B_{\varepsilon}} \int_{B_{\varepsilon}} d_{cc}(f(pz^{-1}), f(pw^{-1}))  d w d z\\
				&\leq  C \cdot \varepsilon^{-\nu_1}    \int \left|D_H f\right|(x)  d x
			\end{aligned}
		\end{equation*}
		then
		\begin{equation*}
			\|X(f_{\varepsilon}^{*} \omega_{1}^k)(p)\|_{L^q}^q \leq \|D_Hf\|_{L^q}^q
		\end{equation*}
		It is worth to mention that the domain $\mathcal{L}_\delta$ can not only be defined for step-2 Carnot group. For higher step Carnot group, we can define $\mathcal{L}_\delta$ to restrict the increase near singular point such that it increases linearly. 
	\end{remark}

	The techniques developed above for step-two Carnot groups-specifically the use of the quasi-metric $d_K$, the Poincaré inequalities (1,1) and (1,2), and the Hardy-Littlewood maximal function estimates-lay the groundwork for the general step-$n$ case.  The convergence rate of the mollified contact forms will degrade proportionally to the step, leading to the critical exponent $q>n+1$. 
	Therefore, for higher step Carnot group $ G_2$, we can get a similar result with more technical and complicated scaling.
	
	\begin{thm}\label{thm-approx-higher}
		Suppose that  $f \in W^{1,q}\left(\Omega ; G_2\right)$ , where $ G_2$ is a step-$n$ Carnot group and $\nu_2$ is the topological dimension of $G_2$, $\Omega \subset G_1$  is open , $X$ is a left-invariant  vector field belonging to the horizontal layer and $q>n+1$ . If  $B\left(x_{o}, 2 r\right) \subset \Omega $, then  for all  $0<\varepsilon<r$,
		\begin{equation}\label{approx-higher}
			D_Hf_\varepsilon \to D_Hf \  \text{in}\  L_{loc}^\frac{q}{n+1}(\Omega;\mathbb{R}^{\nu_2}).\\
		\end{equation}
		where the convergence is in the sense of the norm on Euclidean space $\mathbb{R}^{\nu_2}$, $f_{\varepsilon}=f * \eta_{\varepsilon}=\int_{\mathbb{G}_1 } f(y) \eta_{\varepsilon}\left(y^{-1} x\right) d y=\int_{\mathbb{G}_1 } f(xy^{-1}) \eta_{\varepsilon}\left(y\right) d y$ and $\eta_{\varepsilon}$ is a smooth function on Carnot group $G_1$ with a compact supported set $B(e,\varepsilon)$ and
		\begin{equation*}
			\int_{G_1}\eta_{\varepsilon}=1.
		\end{equation*}
		\begin{proof}
			\begin{equation*}
				\begin{aligned}
					\tilde{\mathcal{C}_{m}}&=\varphi_{2}^m (\left(a_{i}, b_{k}, c_{m}\right),\left(A_{i}, B_{k}, C_{m}\right))\\
					&=c_{m}-C_{m}-\frac{1}{2} \sum_{i, j} \beta_{m}^{i j}\left(A_{i} b_{j}-B_{j} a_{i}\right)+\frac{1}{12} \sum_{l, k} \sum_{i<j}\left(A_{l}+a_{l}\right) \alpha_{k}^{i j}\left(A_{i} a_{j}-a_{i} A_{j}\right) \beta_{m}^{l k}.
				\end{aligned}
			\end{equation*}
			It is easy to find the following result.
			\begin{equation*}
				\begin{aligned}
					&X(f_{\varepsilon}^{*} \omega_{2}^m)+X\left(f_{\varepsilon}^{*}\left(-\frac{1}{2} \sum_{i, j} \beta_{m}^{i j}A_{i} dB_{j}\right)\right)+X\left(f_{\varepsilon}^{*}\left(\frac{1}{12} \sum_{l, k} \sum_{i<j}A_{l} \alpha_{k}^{i j}\left(A_{i} dA_{j}-A_{j} dA_{i}\right) \beta_{m}^{l k} \right)\right)\\
					&=\varepsilon^{-m-1} \int_{B_{\varepsilon}}\int_{B_{\varepsilon}}\left(f^{C_{m}}(pz^{-1})-f^{C_{m}}(pv^{-1})\right)\eta_{\varepsilon}(v)X( \eta(\delta_{\varepsilon^{-1}}(z)))d zdv\\
					&-\frac{1}{2} \sum_{i, j} \beta_{m}^{i j}\varepsilon^{-m-1} \int_{B_{\varepsilon}}\int_{B_{\varepsilon}}\left(f_{\varepsilon}^{A_{i}}(pv^{-1}) f^{B_{j}}(pz^{-1})-f_{\varepsilon}^{B_{j}}(pv^{-1}) f^{A_{i}}(pz^{-1})\right) \\
					&\cdot\eta_{\varepsilon}(v)X( \eta(\delta_{\varepsilon^{-1}}(z)))d zdv\\
					&+\frac{1}{12} \sum_{l, k} \sum_{i<j} \beta_{m}^{l k}\alpha_{k}^{i j}\int_{B_{\varepsilon}}\int_{B_{\varepsilon}}\int_{B_{\varepsilon}}(f^{A_{l}}(pv^{-1})\left(f^{A_{i}}(pv^{-1})f^{A_{j}}(pz^{-1})- f^{A_{j}}(pv^{-1})f^{A_{i}}(pz^{-1})\right)\\ &\cdot\eta_{\varepsilon}(w)\eta_{\varepsilon}(v)X( \eta(\delta_{\varepsilon^{-1}}(z)))d zdvdw\\
					&+\frac{1}{12} \sum_{l, k} \sum_{i<j} \beta_{m}^{l k}\alpha_{k}^{i j}\int_{B_{\varepsilon}}\int_{B_{\varepsilon}}\int_{B_{\varepsilon}}(f^{A_{l}}(pz^{-1})\left(f^{A_{i}}(pv^{-1})f^{A_{j}}(pz^{-1})- f^{A_{j}}(pv^{-1})f^{A_{i}}(pz^{-1})\right)\\ &\cdot\eta_{\varepsilon}(v)\eta_{\varepsilon}(w)X( \eta(\delta_{\varepsilon^{-1}}(z)))d wdzdv\\
					&:=\mathbb{A}
				\end{aligned}
			\end{equation*}
			Since 
			\begin{equation*}
				\begin{aligned}
					&\left|f_\varepsilon^{A_{i}} (p)f_\varepsilon^{A_{j}} (p)-(f^{A_{i}}(p)f^{A_{j}}(p))_\varepsilon\right|\\
					&\le \int_{B_{\varepsilon}}\int_{B_{\varepsilon}}\left|f^{A_{i}}(pz^{-1})\right|\left|f^{A_{j}}(pv^{-1})-f^{A_{j}}(pz^{-1})\right|\eta_\varepsilon(z)\eta_\varepsilon(v)dzdv\\
					&\lesssim
					\int_{B_{\varepsilon}}\int_{B_{\varepsilon}}\left|f(pz^{-1})\right|d_{cc}(f(pv^{-1}),f(pz^{-1}))\eta_\varepsilon(z)\eta_\varepsilon(v)dzdv\\
					&\lesssim \varepsilon^{-2\nu_2+1}  \int_{B_{\varepsilon}}\left|f(pz^{-1})\right|\int_{\lambda B_{\varepsilon}}\left|D_H f\right|dzdv\\
					&\lesssim \varepsilon \mathcal{M}(f)\cdot \mathcal{M}(|D_H f|)
				\end{aligned}
			\end{equation*}
			where we use the $(1,1)$-Poincaré inequality.
			We have
			\begin{equation*}
				\begin{aligned}
					&|\int_{B_{\varepsilon}}\int_{B_{\varepsilon}}\int_{B_{\varepsilon}}f^{A_{l}}(pw^{-1})\left(f^{A_{i}}(pv^{-1})f^{A_{j}}(pz^{-1})- f^{A_{j}}(pv^{-1})f^{A_{i}}(pz^{-1})\right)\\
					&\cdot\eta_{\varepsilon}(v)\eta_{\varepsilon}(w)X( \eta(\delta_{\varepsilon^{-1}}(z)))d zdwdv\\
					&-\int_{B_{\varepsilon}}\int_{B_{\varepsilon}}f^{A_{l}}(pv^{-1})\left(f^{A_{i}}(pv^{-1})f^{A_{j}}(pz^{-1})- f^{A_{j}}(pv^{-1})f^{A_{i}}(pz^{-1})\right)\\
					&\cdot\eta_{\varepsilon}(v)X( \eta(\delta_{\varepsilon^{-1}}(z)))d zdv | \\
					&
					\lesssim \varepsilon \mathcal{M}(f)\cdot \mathcal{M}(|D_H f|) \int_{B_{\varepsilon}}\int_{B_{\varepsilon}}\left|\left(f^{A_{i}}(pv^{-1})f^{A_{j}}(pz^{-1})- f^{A_{j}}(pv^{-1})f^{A_{i}}(pz^{-1})\right)\right|\\
					&\cdot\eta_{\varepsilon}(v)X( \eta(\delta_{\varepsilon^{-1}}(z)))d zdv
				\end{aligned}
			\end{equation*}
			Thus, we have
			\begin{equation*}
				\begin{aligned}
					\mathbb{A}&=\varepsilon^{-\nu_2-1} \int_{B_{\varepsilon}}\int_{B_{\varepsilon}}\varphi_{2}^m (f(pz^{-1}),f(pv^{-1}))\eta_{\varepsilon}(v)X( \eta(\delta_{\varepsilon^{-1}}(z)))dzdv\\
					&+\frac{1}{12}\varepsilon \mathcal{M}(f)\cdot \mathcal{M}(|D_H f|) \sum_{l, k} \sum_{i<j} \beta_{m}^{l k}\alpha_{k}^{i j}\\
					&\cdot\int_{B_{\varepsilon}}\int_{B_{\varepsilon}}\left(f^{A_{i}}(pv^{-1})f^{A_{j}}(pz^{-1})- f^{A_{j}}(pv^{-1})f^{A_{i}}(pz^{-1})\right)\eta_{\varepsilon}(v)X( \eta(\delta_{\varepsilon^{-1}}(z)))d zdv\\
					&\lesssim \varepsilon^{-\nu_2-1} \int_{B_{\varepsilon}}\int_{B_{\varepsilon}}\varphi_{2}^m (f(pz^{-1}),f(pv^{-1}))\eta_{\varepsilon}(v)X( \eta(\delta_{\varepsilon^{-1}}(z)))dzdv\\
					&+\frac{1}{6}\varepsilon \mathcal{M}(f)\cdot \mathcal{M}(|D_H f|) \sum_{l, k}  \beta_{m}^{l k}\\
					&\cdot\int_{B_{\varepsilon}}\int_{B_{\varepsilon}}(\left|\varphi_1^k(f(pz^{-1}),f(pv^{-1}))\right|+\left|f^{B_k}(pz^{-1})-f^{B_k}(pv^{-1})\right| )\eta_{\varepsilon}(v)X( \eta(\delta_{\varepsilon^{-1}}(z)))d zdv\\
					&\lesssim \varepsilon^{-\nu_2-1} \int_{B_{\varepsilon}}\int_{B_{\varepsilon}}\varphi_{2}^m (f(pz^{-1}),f(pv^{-1}))\eta_{\varepsilon}(v)X( \eta(\delta_{\varepsilon^{-1}}(z)))dzdv\\
					&+\frac{1}{6}\varepsilon \mathcal{M}(f)\cdot \mathcal{M}(|D_H f|) \sum_{l, k}  \beta_{m}^{l k}\\
					&\cdot\int_{B_{\varepsilon}}\int_{B_{\varepsilon}}(d_{cc}(f(pz^{-1}),f(pv^{-1}))^2+d_{cc}(f(pz^{-1}),f(pv^{-1})))\eta_{\varepsilon}(v)X( \eta(\delta_{\varepsilon^{-1}}(z)))d zdv\\
					&\lesssim\varepsilon^{-2\nu_2-1} \int_{B_{\varepsilon}}\int_{B_{\varepsilon}}d_{cc}(f(pz^{-1}),f(pv^{-1}))^3dzdv\\ 
					&+\varepsilon^2 \mathcal{M}(f)\cdot \mathcal{M}(|D_H f|) \mathcal{M}(|D_H f|^2)+ \varepsilon \mathcal{M}(f)\cdot \mathcal{M}(|D_H f|) \mathcal{M}(|D_H f|)\\
					&\lesssim\varepsilon^{2}\mathcal{M}(|D_H f|^3)  
					+\varepsilon^2 \mathcal{M}(f) \mathcal{M}(|D_H f|) \mathcal{M}(|D_H f|^2)+ \varepsilon \mathcal{M}(f) \mathcal{M}(|D_H f|) \mathcal{M}(|D_H f|)
				\end{aligned}
			\end{equation*}
			where we use the $(1,1)$, $(1,2)$ and $(1,3)$-Poincaré inequality.
			Thus,
			\begin{equation*}
				\begin{array}{c}
					\left|X(f_{\varepsilon}^{*} \omega_{2}^m)+X\left(f_{\varepsilon}^{*}\left(-\frac{1}{2} \sum_{i, j} \beta_{m}^{i j}A_{i} dB_{j}\right)\right)+X\left(f_{\varepsilon}^{*}\left(\frac{1}{12} \sum_{l, k} \sum_{i<j}A_{l} \alpha_{k}^{i j}\left(A_{i} dA_{j}-A_{j} dA_{i}\right) \beta_{m}^{l k} \right)\right)\right|\\
					\lesssim\varepsilon^{2}\mathcal{M}(|D_H f|^3)  
					+\varepsilon^2 \mathcal{M}(f) \mathcal{M}(|D_H f|) \mathcal{M}(|D_H f|^2)+ \varepsilon \mathcal{M}(f) \mathcal{M}(|D_H f|) \mathcal{M}(|D_H f|)
				\end{array}
			\end{equation*}
			then we can get
			\begin{equation*}
				\begin{array}{c}
					\left|X(f_{\varepsilon}^{*} \omega_{2}^m)\right|\lesssim \left(\varepsilon^{2}\mathcal{M}(|D_H f|^3)  
					+\varepsilon^2 \mathcal{M}(f) \mathcal{M}(|D_H f|) \mathcal{M}(|D_H f|^2)+ \varepsilon \mathcal{M}(f) \mathcal{M}(|D_H f|) \mathcal{M}(|D_H f|)\right)\\
					+X\left(f_{\varepsilon}^{*}\left(-\frac{1}{2} \sum_{i, j} \beta_{m}^{i j}A_{i} dB_{j}\right)\right)+X\left(f_{\varepsilon}^{*}\left(\frac{1}{12} \sum_{l, k} \sum_{i<j}A_{l} \alpha_{k}^{i j}\left(A_{i} dA_{j}-A_{j} dA_{i}\right) \beta_{m}^{l k} \right)\right)\\
					\lesssim \left(\varepsilon^{2}\mathcal{M}(|D_H f|^3)  
					+\varepsilon^2 \mathcal{M}(f) \mathcal{M}(|D_H f|) \mathcal{M}(|D_H f|^2)+ \varepsilon \mathcal{M}(f) \mathcal{M}(|D_H f|) \mathcal{M}(|D_H f|)\right)\\
					+\left(\frac{5}{12} \sum_{l, k} \sum_{i<j} \beta_{m}^{l k}f_{\varepsilon}^{A(i)}(p)X f_{\varepsilon}^{B(j)}(p) \right)+\left(\frac{1}{12} \sum_{l, k} \sum_{i<j} \beta_{m}^{l k}f_{\varepsilon}^{A(i)}(p)\varepsilon\mathcal{M}(f)\mathcal{M}(|D_Hf|) \right)
				\end{array}
			\end{equation*}
			For $X(f_\varepsilon^{B(j)})(p)$, on a compact subset,
			\begin{equation*}
				\begin{aligned}
					X(f_\varepsilon^{B(j)})(p)&\leq \varepsilon^{-\nu_1-1}\int_{B_{\varepsilon}}\left(f^{B_j}(pz^{-1})-f^{B_j}(p)\right)X( \eta(\delta_{\varepsilon^{-1}}(z)))d z\\
					& \lesssim \varepsilon^{-1}\int_{B_{\varepsilon}} d_K(f(pz^{-1}),f(p))^2dz\\
					& \lesssim \varepsilon^{-1}\int_{B_{\varepsilon}} d_{cc}(f(pz^{-1}),f(p))^2dz\\
					&\leq \varepsilon\mathcal{M}(|D_Hf|^2)(p)
				\end{aligned}
			\end{equation*}
			where we use the $(1,2)$-Poincaré inequality.
			We know that 
			\begin{equation*}
				\begin{aligned}
					\left|X(f_\varepsilon^{B(j)})(p)\right|&\lesssim  \left(\varepsilon^{2}\mathcal{M}(|D_H f|^3)  
					+\varepsilon^2 \mathcal{M}(f) \mathcal{M}(|D_H f|) \mathcal{M}(|D_H f|^2)+ \varepsilon \mathcal{M}(f) \mathcal{M}(|D_H f|) \mathcal{M}(|D_H f|)\right)\\
					&+\varepsilon \mathcal{M}(f) \mathcal{M}(|D_H f|^2) +\varepsilon \mathcal{M}(f)\mathcal{M}(f) \mathcal{M}(|D_H f|)
				\end{aligned}				
			\end{equation*}
			and 
			\begin{equation*}
				\begin{aligned}
					\|X(f_{\varepsilon}^{*} \omega_{2}^m)\|_{L^\frac{q}{4}(W)}^\frac{q}{4}&\lesssim  \varepsilon^{2}\|\mathcal{M}(|D_H f|^3)\|_{L^\frac{q}{4}(W)}^\frac{q}{4}  
					+\varepsilon^2\| \mathcal{M}(f) \mathcal{M}(|D_H f|) \mathcal{M}(|D_H f|^2)\|_{L^\frac{q}{4}(W)}^\frac{q}{4}\\
					&+ \varepsilon \|\mathcal{M}(f) \mathcal{M}(|D_H f|) \mathcal{M}(|D_H f|)\|_{L^\frac{q}{4}(W)}^\frac{q}{4}+\varepsilon \|\mathcal{M}(f) \mathcal{M}(|D_H f|^2)\|_{L^\frac{q}{4}(W)}^\frac{q}{4} \\
					&+\varepsilon \|\mathcal{M}(f)\mathcal{M}(f) \mathcal{M}(|D_H f|)\|_{L^\frac{q}{4}(W)}^\frac{q}{4}\\
					&\lesssim \varepsilon^{2}\||D_H f|\|_{L^\frac{3q}{4}(W)}^\frac{q}{4}  
					+\varepsilon^2\| \mathcal{M}(f)\|_{L^{q}(W)}^\frac{q}{4}\| \mathcal{M}(|D_H f|)\|_{L^{q}(W)}^\frac{q}{4}  \| \mathcal{M}(|D_H f|^2)\|_{L^\frac{q}{2}(W)}^\frac{q}{4}\\
					&+ \varepsilon \|\mathcal{M}(f)\|_{L^\frac{q}{2}(W)}^\frac{q}{4} \|\mathcal{M}(|D_H f|)\|_{L^{q}(W)}^\frac{q}{2} +\varepsilon \|\mathcal{M}(f)\|_{L^\frac{q}{2}(W)}^\frac{q}{4} \| \mathcal{M}(|D_H f|^2)\|_{L^\frac{q}{2}(W)}^\frac{q}{4} \\
					&+\varepsilon \|\mathcal{M}(f)\|_{L^q(W)}^\frac{q}{2}\| \mathcal{M}(|D_H f|)\|_{L^\frac{q}{2}(W)}^\frac{q}{4}\\
					&\lesssim \varepsilon^{2}\||D_H f|\|_{L^\frac{3q}{4}(W)}^\frac{q}{4}  
					+\varepsilon^2\| f\|_{L^{q}(W)}^\frac{q}{4}\| |D_H f|\|_{L^{q}(W)}^\frac{q}{4}  \| |D_H f|\|_{L^{q}(W)}^\frac{q}{2}\\
					&+ \varepsilon \|f\|_{L^\frac{q}{2}(W)}^\frac{q}{4} \||D_H f|\|_{L^{q}(W)}^\frac{q}{2} +\varepsilon \|f\|_{L^\frac{q}{2}(W)}^\frac{q}{4} \| |D_H f|\|_{L^{q}(W)}^\frac{q}{2} \\
					&+\varepsilon \|f\|_{L^q(W)}^\frac{q}{2}\| |D_H f|\|_{L^\frac{q}{2}(W)}^\frac{q}{4}\\
					& \rightarrow 0,
				\end{aligned}				
			\end{equation*}
			for the second inequality, we take advantage of Hölder inequality.
			
			For step-$n$ Carnot groups, we can similarly prove that $X(f_{\varepsilon}^{*} \omega_{3}^k)\to 0 ,
			\dots ,
			X(f_{\varepsilon}^{*} \omega_{n_2}^m)  \to 0 $ in $L_{loc}^\frac{q}{n+1}$. We leave it to readers. 
		\end{proof}	
	\end{thm}
	
	While Theorem \ref{thm-approx-higher} establishes the convergence of the horizontal part of the differential $D_H f_\varepsilon$, the behavior of vector fields acting on pullbacks of higher-layer forms remains to be characterized. 
	
	\begin{thm}\label{thm-approx-second layer}
		Suppose that  $f \in W^{1,q}\left(\Omega ; G_2\right)$ , where $ G_2$ is a step-2 Carnot group and $\nu_2$ is the topological  dimension of $G_2$, $\Omega \subset G_1$  is open , $X$ is a left-invariant  vector field belonging to the horizontal layer and $q\geq2$ . If  $B\left(x_{o}, 2 r\right) \subset \Omega $, then  for all  $0<\varepsilon<r$,
		\begin{equation}\label{approx}
			\|Yf_\varepsilon^*(\omega_k)\|_{L^\frac{q}{2}(\Omega)}^\frac{q}{2} \lesssim \|D_Hf\|_{L^q(\Omega)}^\frac{q}{2} 
		\end{equation}
		where the convergence is in the sense of the norm on Euclidean space $\mathbb{R}^{\nu_2}$, $f_{\varepsilon}=f * \eta_{\varepsilon}=\int_{\mathbb{G}_1 } f(y) \eta_{\varepsilon}\left(y^{-1} x\right) d y$  and $\eta_{\varepsilon}$ is a smooth function on Carnot group $G_1$ with a compact supported set $B(e,\varepsilon)$ and
		\begin{equation*}
			\int_{G_1}\eta_{\varepsilon}=1.
		\end{equation*}
		\begin{proof}
			Recall that  $\varphi(p, q)$  was defined in \eqref{metric-K}. Let  $p \in B\left(x_{o}, r\right)$  and let  $\varepsilon \in(0, r)$ . In what follows we will identify  $\left(f_{\varepsilon}^{*} \alpha\right)(p)$  with the vector (of equal length):
			
			\begin{equation*}
				\begin{aligned}
					&Y(f_{\varepsilon}^{*} \omega_1^k)(p)\\
					=&Y( f_{\varepsilon}^{B_k})(p)-\frac{1}{2} \sum_{i<j}\alpha_k^{ij} \left(f_{\varepsilon}^{A_i}(p) Y(f_{\varepsilon}^{A_j})(p)-f_{\varepsilon}^{A_j}(p) Y( f_{\varepsilon}^{A_i})(p)\right)\\
					=&
					\varepsilon^{-\nu_1-2} \int_{B_{\varepsilon}}\left(f^{B_k}(pz^{-1})-f_{\varepsilon}^{B_k}(p)+\frac{1}{2} \sum_{i=1}^n\left(f_{\varepsilon}^{A_i}(p) f^{A_j}(pz^{-1})-f_{\varepsilon}^{A_j}(p) f^{A_i}(pz^{-1})\right)\right)\\
					& \cdot  Y(\eta\left(\delta_{\varepsilon^{-1}}(z)\right)) d z \\
					=&\varepsilon^{-\nu_1-2} \int_{B_{\varepsilon}} \varphi_{k}\left(f(pz^{-1}), f_{\varepsilon}(p)\right) Y(\eta\left(\delta_{\varepsilon^{-1}}(z)\right)) d z \\
					=&\varepsilon^{-\nu_1-2} \int_{B_{\varepsilon}} \int_{B_{\varepsilon}} \varphi_k(f(pz^{-1}), f(pw^{-1})) \eta_{\varepsilon}(w) Y(\eta\left(\delta_{\varepsilon^{-1}}(z)\right)) d w d z
				\end{aligned}
			\end{equation*}
			Here the second equality uses the fact that
			$\int_{B_{\varepsilon}} f_{\varepsilon}^{B_k}(p)\, Y\big(\eta(\delta_{\varepsilon^{-1}}(z))\big) dz=0$,
			and the last one is obtained by writing the value $f_\varepsilon^{B_k}(p)$ itself
			as the average $\int_{B_\varepsilon}f^{B_k}(pw^{-1})\eta_\varepsilon(w)\,dw$.
			Note that \eqref{metric-K} yields
			\begin{equation*}
				|\varphi(f(pz^{-1}), f(pw^{-1}))| \leq d_{K}(f(pz^{-1}), f(pw^{-1}))^{2} 
			\end{equation*}
			Therefore,
			\begin{equation*}
				\begin{aligned}
					\left|Y(f_{\varepsilon}^{*} \omega)(p)\right| 
					&\leq\varepsilon^{-\nu_1-2} \int_{B_{\varepsilon}} \int_{B_{\varepsilon}} \varphi_k(f(pz^{-1}), f(pw^{-1})) \eta_{\varepsilon}(w) Y(\eta\left(\delta_{\varepsilon^{-1}}(z)\right)) d w d z\\
					&\leq \varepsilon^{-\nu_1-2} \int_{B_{\varepsilon}} \int_{B_{\varepsilon}} d_{K}(f(pz^{-1}), f(pw^{-1}))^{2} \eta_{\varepsilon}(w) Y(\eta\left(\delta_{\varepsilon^{-1}}(z)\right)) d w d z\\
					&\lesssim \varepsilon^{-\nu_1-2} \int_{B_{\varepsilon}} \int_{B_{\varepsilon}} d_{cc}(f(pz^{-1}), f(pw^{-1}))^{2} \eta_{\varepsilon}(w) Y(\eta\left(\delta_{\varepsilon^{-1}}(z)\right)) d w d z\\
					&\lesssim \varepsilon^{-2 \nu_1-2} \int_{B_{\varepsilon}} \int_{B_{\varepsilon}} d_{cc}(f(pz^{-1}), f(pw^{-1}))^{2}  d w d z\\
				\end{aligned}
			\end{equation*}
			From here on the argument is identical to the one in the proof of
			Theorem \ref{thm-approx}, the only difference being that the
			horizontal field $X$ is replaced by the vertical field $Y$; since $Y$ is
			homogeneous of degree $2$ with respect to the dilations, this costs one
			extra power of $\varepsilon$, and we obtain
			\begin{equation*}
				\| Yf_\varepsilon^*(\omega_k)\|_{L^\frac{q}{2}(\Omega)}^\frac{q}{2}\leq \int_{\Omega} |\mathcal{M}(|D_Hf|^2)|^\frac{q}{2}\lesssim \|D_Hf\|_{L^q(\Omega)}^\frac{q}{2}.
			\end{equation*}

		\end{proof}	
	\end{thm}	
	
	\begin{thm}\label{thm-approx-first layer by second fector}
		Suppose that  $f \in W^{1,q}\left(\Omega ; G_2\right)$ , where $ G_2$ is a step-2 Carnot group and $\nu_2$ is the topological  dimension of $G_2$, $\Omega \subset G_1$  is open , $X$ is a left-invariant  vector field belonging to the horizontal layer and $q\geq2$ . If  $B\left(x_{o}, 2 r\right) \subset \Omega $, then  for all  $0<\varepsilon<r$,
		\begin{equation}\label{approx}
			\|Yf_\varepsilon^*(dx_k)\|_{L^q(\Omega)}^q \lesssim \frac{1}{\varepsilon}\|D_Hf\|_{L^q(\Omega)}^q 
		\end{equation}
		where the convergence is in the sense of the norm on Euclidean space $\mathbb{R}^{\nu_2}$, $f_{\varepsilon}=f * \eta_{\varepsilon}=\int_{\mathbb{G}_1 } f(y) \eta_{\varepsilon}\left(y^{-1} x\right) d y$  and $\eta_{\varepsilon}$ is a smooth function on Carnot group $G_1$ with a compact supported set $B(e,\varepsilon)$ and
		\begin{equation*}
			\int_{G_1}\eta_{\varepsilon}=1.
		\end{equation*}
		\begin{proof}
			Recall that  $\varphi(p, q)$  was defined in \eqref{metric-K}. Let  $p \in B\left(x_{o}, r\right)$  and let  $\varepsilon \in(0, r)$ . In what follows we will identify  $\left(f_{\varepsilon}^{*} \alpha\right)(p)$  with the vector (of equal length):
			
			\begin{equation*}
				\begin{aligned}
					&Y(f_{\varepsilon}^{*} dx_k)(p)\\
					=&Y( f_{\varepsilon}^{A_k})(p)\\
					=&
					\varepsilon^{-\nu_1-2} \int_{B_{\varepsilon}}\left(f^{A_k}(pz^{-1})-f_{\varepsilon}^{A_k}(p)\right)\cdot  Y(\eta\left(\delta_{\varepsilon^{-1}}(z)\right)) d z \\
					\leq&\varepsilon^{-\nu_1-2} \int_{B_{\varepsilon}} \int_{B_{\varepsilon}} \left(f^{A_k}(pz^{-1})-f_{\varepsilon}^{A_k}(p)\right) \eta_{\varepsilon}(w) Y(\eta\left(\delta_{\varepsilon^{-1}}(z)\right)) d w d z\\
					\leq&\varepsilon^{-\nu_1-2} \int_{B_{\varepsilon}} \int_{B_{\varepsilon}} d_K(f(pz^{-1}), f(pw^{-1})) \eta_{\varepsilon}(w) Y(\eta\left(\delta_{\varepsilon^{-1}}(z)\right)) d w d z
				\end{aligned}
			\end{equation*}
			Here the second equality uses the fact that
			$\int_{B_{\varepsilon}} f_{\varepsilon}^{A_k}(p)\, Y\big(\eta(\delta_{\varepsilon^{-1}}(z))\big) dz=0$,
			and the last one is obtained by writing the value $f_\varepsilon^{A_k}(p)$ itself
			as the average $\int_{B_\varepsilon}f^{A_k}(pw^{-1})\eta_\varepsilon(w)\,dw$.
			Note that \eqref{metric-K} yields
			\begin{equation*}
				|\pi f(pz^{-1})-\pi f(pw^{-1}))| \leq d_{K}(f(pz^{-1}), f(pw^{-1}))
			\end{equation*}
			Therefore,
			\begin{equation*}
				\begin{aligned}
					&Y(f_{\varepsilon}^{*} dx_k)(p)\\
					\leq&\varepsilon^{-\nu_1-2} \int_{B_{\varepsilon}} \int_{B_{\varepsilon}} d_K(f(pz^{-1}), f(pw^{-1})) \eta_{\varepsilon}(w) Y(\eta\left(\delta_{\varepsilon^{-1}}(z)\right)) d w d z\\
					\lesssim&\varepsilon^{-\nu_1-2} \int_{B_{\varepsilon}} \int_{B_{\varepsilon}} d_{cc}(f(pz^{-1}), f(pw^{-1})) \eta_{\varepsilon}(w) Y(\eta\left(\delta_{\varepsilon^{-1}}(z)\right)) d w d z\\
					\lesssim&\varepsilon^{-2\nu_1-2} \int_{B_{\varepsilon}(p)} \int_{B_{\varepsilon}(p)} d_{cc}(f(x), f(y)) dx dy\\
					\leq& C \varepsilon^{-2\nu_1-1}    \int_{B_{\varepsilon}(p)} \int_{\lambda B_{\varepsilon}(y)} |D_H f|dx dy\\
					\lesssim& \frac{1}{\varepsilon}\mathcal{M}(\left|D_H f\right|)(p) 
				\end{aligned}
			\end{equation*}
			by the theory of maximal function,
			\begin{equation*}
				\|Yf_\varepsilon^*(dx_k)\|_{L^q(\Omega)}^q \lesssim \frac{1}{\varepsilon}\|D_Hf\|_{L^q(\Omega)}^q 
			\end{equation*}
			
		\end{proof}	
	\end{thm}

	\begin{thm}\label{Pullback}
		Suppose that  $f \in W^{1,q}\left(\Omega ; G_2\right)$, where $ G_2$ is a step-$2$ Carnot group and $\nu_2$ is the topological dimension of $G_2$, $\Omega \subset G_1$  is open and homogenuous dimension of $G_1$ is $\nu_1$, topological dimension of $G_1$ is $N$. For their Lie algebras, $dim(g_1^{[1]})\geq dim(g_2^{[1]})$ and $dim(g_1^{[2]})\leq dim(g_2^{[2]})$. If  $B\left(x_{o}, 2 r\right) \subset \Omega $, then  for all  $0<\varepsilon<r$ and $q\geq |wt(\theta)|$,
		\begin{equation}
			f_\varepsilon^*(\theta)\wedge \eta \to f^*(\theta)\wedge \eta \  \text{in}\  L_{loc}^{1}(\Omega;\mathbb{R}^{\nu_2}).\\
		\end{equation}
		where $\theta \in \Omega^k(U')$, $\eta \in \Omega_c^{N-k}(U)$ and $\theta$ contains all the contact forms on Carnot group $G_2$.
	\end{thm}	
	\begin{proof}
		Suppose that $\theta=\theta_1 \wedge \theta_2$ where $\theta_1=dx_1\wedge dx_2\wedge\dots \wedge dx_{k_1}$ is the component in the first layer and $\theta_2=\omega_1\wedge \dots \wedge\omega_{k_2}$ is the component in the second layer. Then
		\begin{equation*}
			\begin{aligned}
				&\|f_\varepsilon^*(\theta)\wedge\eta\|_{L^1(\Omega)}\\
				&=\int_\Omega \left|f_\varepsilon^*(\theta_1) \wedge f_\varepsilon^*(\theta_2)\wedge \eta\right|\\
				&=\int_\Omega \left|f_\varepsilon^*(dx_1\wedge dx_2\wedge\dots \wedge dx_{k_1}) \wedge f_\varepsilon^*(\omega_1\wedge \dots \wedge\omega_{k_2})\wedge \eta\right|\\
				&=\int_\Omega \left|\wedge_{i=1}^{k_1}\left(X_1f_\varepsilon^*(dx_i)dx_1+\dots +X_{dim(g_{1}^{[1]})}f_\varepsilon^*(dx_i)dx_{dim(g_{1}^{[1]})}\right.\right.\\
				&\left.+Y_1f_\varepsilon^*(dx_i)\omega_1+\dots+Y_{dim(g_1^{[1]})}f_\varepsilon^*(dx_i)\omega_{dim(g_1^{[1]})}\right)\\ 
				&\left.\wedge_{j=1}^{k_2}\left(X_1f_\varepsilon^*(\omega_j)dx_1+\dots +X_{dim(g_{1}^{[1]})}f_\varepsilon^*(\omega_j)dx_{dim(g_{1}^{[1]})}\right.\right.\\
				&\left.\left.+Y_1f_\varepsilon^*(\omega_j)\omega_1+\dots+Y_{dim(g_1^{[1]})}f_\varepsilon^*(\omega_j)\omega_{dim(g_1^{[1]})}\right)\wedge\eta\right|\\ 
			\end{aligned}
		\end{equation*}
		We know that some term in the equation above would explode when $\varepsilon \to 0$, thus, for convenience, we would rearrange in order of size
		\begin{equation*}
			\begin{aligned}
				&\|f_\varepsilon^*(\theta)\wedge\eta\|_{L^1(\Omega)}\\
				=& \|\eta\|_{L^\infty}\int_\Omega\sum_{\text{same order}} \left( Y_1f_\varepsilon^*(dx_1) \dots  Y_{k_1}f_\varepsilon^*(dx_{k_1})\right)\\
				&\cdot \left(X_1f_\varepsilon^*(\omega_1) \dots  X_{k_1+k_2-dim(g_1^{[2]})}f_\varepsilon^*(\omega_{k_1+k_2-dim(g_1^{[2]})}) \right)\\
				&\cdot
				\left( Y_1f_\varepsilon^*(\omega_1) \dots  Y_{dim(g_1^{[2]})}f_\varepsilon^*(\omega_{dim(g_1^{[2]})})\right)\\
				&+\|\eta\|_{L^\infty}\int_\Omega\sum_{\text{same order}} \left( X_1f_\varepsilon^*(dx_1) \dots  Y_{k_1}f_\varepsilon^*(dx_{k_1})\right)\\
				&\cdot \left(X_1f_\varepsilon^*(\omega_2) \dots  X_{k_1+k_2-dim(g_1^{[2]})+1}f_\varepsilon^*(\omega_{k_1+k_2-dim(g_1^{[2]})+1}) \right)\\
				&\cdot
				\left( Y_1f_\varepsilon^*(\omega_1) \dots  Y_{dim(g_1^{[2]})-1}f_\varepsilon^*(\omega_{dim(g_1^{[2]})-1})\right)\\
				&+\dots\\
				&\lesssim \left(\|Yf_\varepsilon^*(dx)\|_{L^q}^q\right)^{k_1} \cdot  \left(\|Xf_\varepsilon^*(\omega)\|_{L^\frac{q}{2}}^\frac{q}{2}\right)^{k_1+k_2-dim(g_1^{[2]})}\cdot  \left(\|Yf_\varepsilon^*(\omega)\|_{L^\frac{q}{2}}^\frac{q}{2}\right)^{dim(g_1^{[2]})}\cdot \|\eta\|_{L^\infty}\\
				&+\left(\|Yf_\varepsilon^*(dx)\|_{L^q}^q\right)^{k_1} \cdot  \left(\|Xf_\varepsilon^*(\omega)\|_{L^\frac{q}{2}}^\frac{q}{2}\right)^{k_1+k_2-dim(g_1^{[2]})+1}\cdot  \left(\|Yf_\varepsilon^*(\omega)\|_{L^\frac{q}{2}}^\frac{q}{2}\right)^{dim(g_1^{[2]})-1}\cdot \|\eta\|_{L^\infty}\\
				&+\left(\|Yf_\varepsilon^*(dx)\|_{L^q}^q\right)^{k_1} \cdot  \left(\|Xf_\varepsilon^*(\omega)\|_{L^\frac{q}{2}}^\frac{q}{2}\right)^{k_1+k_2-dim(g_1^{[2]})+2}\cdot  \left(\|Yf_\varepsilon^*(\omega)\|_{L^\frac{q}{2}}^\frac{q}{2}\right)^{dim(g_1^{[2]})-2}\cdot \|\eta\|_{L^\infty}\\
				&+ \dots \\
				&\lesssim \|\eta\|_{L^\infty} \|D_Hf \|_{L^q}^{q(k_1+\frac{k_1+k_2}{2})} \frac{1}{\varepsilon^{k_1}}\cdot \varepsilon^{k_1+k_2-dim(g_1^{[2]})}\\
				&+\|\eta\|_{L^\infty} \|D_Hf \|_{L^q}^{q(k_1+\frac{k_1+k_2}{2})} \frac{1}{\varepsilon^{k_1}}\cdot \varepsilon^{k_1+k_2-dim(g_1^{[2]})+1}\\
				&+\|\eta\|_{L^\infty} \|D_Hf \|_{L^q}^{q(k_1+\frac{k_1+k_2}{2})} \frac{1}{\varepsilon^{k_1}}\cdot \varepsilon^{k_1+k_2-dim(g_1^{[2]})+2}\\
				&+\dots\\
				&<\infty
			\end{aligned}
		\end{equation*}
		since $k_2=dim(g_2^{[2]})\geq dim(g_1^{[2]})$. Especially, $\|f_\varepsilon^*(\theta)\wedge\eta\|_{L^1(\Omega)}\to 0$ does not necessarily hold when $dim(g_2^{[2]})= dim(g_1^{[2]})$. From Theorem \ref{thm-approx}, \ref{thm-approx-second layer}, \ref{thm-approx-first layer by second fector}, we know that Maximal functions $\mathcal{M}(|D_Hf|)(p)$ and $\mathcal{M}(|D_Hf|^2)(p)$ are the control functions for Dominated Convergence Theorem.
		Therefore,
		\begin{equation}
			\int_{\Omega}f^*(\theta)\wedge \eta:=\int_{\Omega}\lim_{\varepsilon\to 0}f_\varepsilon^*(\theta)\wedge \eta=\lim_{\varepsilon\to 0}\int_{\Omega}f_\varepsilon^*(\theta)\wedge \eta  
		\end{equation}  
	\end{proof}
	
	\begin{thm}\label{Pullback2}
		Suppose that  $f \in W^{1,q}\left(\Omega ; G_2\right)$, where $ G_2$ is a step-$2$ Carnot group and $\nu_2$ is the topological dimension of $G_2$, $\Omega \subset G_1$  is open and homogenuous dimension of $G_1$ is $\nu_1$, topological dimension of $G_1$ is $N$. For their Lie algebras, $dim(g_1^{[1]})= dim(g_2^{[1]})$ and $dim(g_1^{[2]})= dim(g_2^{[2]})$. If  $B\left(x_{o}, 2 r\right) \subset \Omega $, then  for all  $0<\varepsilon<r$ and $q\geq |wt(\theta)|$,
		\begin{equation}
			f_\varepsilon^*(\theta)\wedge \eta \to f^*(\theta)\wedge \eta \  \text{in}\  L_{loc}^{1}(\Omega;\mathbb{R}^{\nu_2}).\\
		\end{equation}
		where $\theta \in \Omega^k(U')$, $\eta \in \Omega_c^{N-k}(U)$ and $wt(\theta)+wt(\eta)\leq \nu_1$.
	\end{thm}	
	\begin{proof}
		Suppose that $dim(g_1^{[1]})= dim(g_2^{[1]})=d_1$ and $dim(g_1^{[2]})= dim(g_2^{[2]})=d_2$, $\theta=\theta_1 \wedge \theta_2$ where $\theta_1=dx_1\wedge dx_2\wedge\dots \wedge dx_{k_1}$ and $\theta_2=\omega_1\wedge \dots \wedge\omega_{k_2}$, $\eta=\varphi d\tilde{x}_1 \wedge \dots \wedge d\tilde{x}_{d_1-k_1}\wedge d\tilde{\omega}_1 \wedge \dots\wedge d\tilde{\omega}_{d_2-k_2}$. For $f_\varepsilon^*(dx_i)$ or $f_\varepsilon^*(\omega_i)$ in the following equation,
		\begin{equation*}
			\begin{aligned}
				&\|f_\varepsilon^*(\theta)\wedge\eta\|_{L^1(\Omega)}\\
				&=\int_\Omega \left|f_\varepsilon^*(\theta_1) \wedge f_\varepsilon^*(\theta_2)\wedge \eta\right|\\
				&=\int_\Omega \left|f_\varepsilon^*(dx_1\wedge dx_2\wedge\dots \wedge dx_{k_1}) \wedge f_\varepsilon^*(\omega_1\wedge \dots \wedge\omega_{k_2})\wedge \eta\right|\\
			\end{aligned}
		\end{equation*}
		It is easy to check that $f_\varepsilon^*(dx_1\wedge dx_2\wedge\dots \wedge dx_{k_1}) \wedge f_\varepsilon^*(\omega_1\wedge \dots \wedge\omega_{k_2})\wedge \eta$ would only remains those factors that contain a volume form on $G_1$. That every such factor matches a permutation $\sigma\in S_n$ is the symmetric group with respect to subscripts. Thus, in each factor the number of terms with increasing weight $(dx_1,dx_2,\dots, dx_{k_1}) \stackrel{f_\varepsilon^*}{\longrightarrow} (\tilde{\omega}_1,\tilde{\omega}_2,\dots, \tilde{\omega}_n, d\tilde{x}_{n+1},\dots ,d\tilde{x}_{k_1})$ is the same as  the number of terms with decreasing weight $(\omega_1,\omega_2,\dots, \omega_{k_2}) \stackrel{f_\varepsilon^*}{\longrightarrow} (d\tilde{x}_1, d\tilde{x}_2,\dots, d\tilde{x}_n, \tilde{\omega}_{n+1},\dots , \tilde{\omega}_{k_2})$. Then the number of $\tilde{T}f_\varepsilon^*(dx)$ would match perfectly with the number of $\tilde{X}f_\varepsilon^*(\omega)$ and every factor in $\int_\Omega \left|f_\varepsilon^*(\theta_1) \wedge f_\varepsilon^*(\theta_2)\wedge \eta\right|$ would not explode as $\varepsilon \to 0$.
		Therefore,
		\begin{equation*}
			\|f_\varepsilon^*(\theta)\wedge\eta\|_{L^1(\Omega)}
			\lesssim \|\varphi\|_{L^\infty} \|D_Hf \|_{L^q}^{q(k_1+\frac{k_1+k_2}{2})} 
			<\infty
		\end{equation*}
		From Theorem \ref{thm-approx}, \ref{thm-approx-second layer}, \ref{thm-approx-first layer by second fector}, we know that Maximal functions $\mathcal{M}(|D_Hf|)(p)$ and $\mathcal{M}(|D_Hf|^2)(p)$ are the control functions for Dominated Convergence Theorem.
		Therefore,
		\begin{equation}
			\int_{\Omega}f^*(\theta)\wedge \eta:=\int_{\Omega}\lim_{\varepsilon\to 0}f_\varepsilon^*(\theta)\wedge \eta=\lim_{\varepsilon\to 0}\int_{\Omega}f_\varepsilon^*(\theta)\wedge \eta  
		\end{equation}  
	\end{proof}
	
	Therefore, we can generalize the classical pullback theorem.
	\begin{thm}\label{thm-Pull-0}
		Let $G_1$ and $G_2$ be step-2 Carnot groups, where $G_1$ has topological dimension $N$ and homogenuous dimension $\nu_1$. Suppose $f \in W^{1,q}(U;G_2)$ and $f:U\to U' \subset G_2$ where $U$ is a open domain in $G_1$ and $\omega \in \Omega^k(U')$ is a continuous closed form, $\eta \in \Omega_c^{N-k-1}(U)$ is smooth. When $p\geq  |wt(\omega)|$, then 
		\begin{equation}
			\int_U f^*(\theta) \wedge d\eta=0
		\end{equation}  
		where $wt(\theta)+wt(d\eta)\leq \nu_1$ or $\theta$ contains all the contact forms on Carnot group $G_2$.
	\end{thm}
	\begin{proof}
		With theorem \ref{Pullback} and \ref{Pullback2}, we have
		
		\begin{equation*}
			\begin{aligned}
				\int_U f^*(\theta)\wedge d\eta 
				&= \lim_{\varepsilon}\int_U f_\varepsilon^*(\theta)\wedge d\eta\\
				&= \lim_{\varepsilon}(-1)^{deg\theta}\int_U df_\varepsilon^*(\theta)\wedge \eta\\
				&= \lim_{\varepsilon}(-1)^{deg\theta}\int_U f_\varepsilon^*(d\theta)\wedge \eta\\
				&= 0,
			\end{aligned}
		\end{equation*}
		since $\theta$ is closed.
	\end{proof}
	
	From Theorem \ref{thm-approx}, \ref{thm-approx-second layer}, \ref{thm-approx-first layer by second fector}, we know that the classical differentials of the mappings with Pansu differential have the following form almost everywhere:
	\begin{equation}\label{diff}
		Df=\begin{bmatrix}
			D_{V_1}f&* \\
			& D_{V_2}f\\
		\end{bmatrix}
	\end{equation}  
	where $D_{V_1}f=D_Hf$ and $D_{V_2}f=\left(D_Pf\right)|_{V_2}$.
	Consequently, if a contact Sobolev mapping is Pansu differentiable almost everywhere, then $f_P^*(\theta)\wedge \eta=f^*(\theta)\wedge \eta$ a.e., when $\theta$ contains all the contact forms on Carnot group $G_2$ or $wt(\theta)+wt(\eta)\leq \nu_1$.
	
	\section{Rigidity of Sobolev mappings on Carnot group}\label{sec:rigidity}
	
	The pullback machinery developed in Section 3 provides a robust substitute for pointwise Pansu differentiability when $p < \nu_1$. Notably, Theorem \ref{thm-Pull-0} allows us to integrate $f_P^* \omega$ against exact forms $d\eta$ and pass limits via mollifications. This capability is crucial for constraining the algebraic structure of the Pansu differential $D_P f$. We now demonstrate how these integral identities force a factorization of the mapping $f$ into products of maps between isomorphic factors, thereby generalizing Xie's rigidity to the low-integrability regime.
	
	Let $\{G_{i}\}_{i\in I}$, $\{G_{j}^{\prime}\}_{j\in I^{\prime}}$ be finite collections of Carnot groups, where each $G_{i}$, $G_{j}^{\prime}$ is nonabelian and does not admit a nontrivial decomposition as a product of Carnot groups. Let $G:=\prod_{i\in I}G_{i}$, $G^{\prime}:=\prod_{j\in I^{\prime}}G_{j}^{\prime}$, and $\mathfrak{g}$, $\mathfrak{g}^{\prime}$ be the graded Lie algebras.
	
	\begin{thm}\label{thm-Rigidity}
		If $U\subset G$ is open and $f:U\to G^{\prime}$ is a Sobolev mapping such that\\
		(1) \ $f\in W^{1,p}_{\mathrm{loc}}$ for $p>max\left\{\text{homogeneous dim of}\ G_i\right\}-1$ .\\
		(2) \ $D_{P}f(x)$ exists and is an isomorphism for $\mu$-a.e. $x\in V$.\\
		Then up to reindexing the collections $\{G_{i}\}_{i\in I}$, $\{G_{j}^{\prime}\}_{j\in I^{\prime}}$ are the same up to isomorphism, and $f$ is locally a product of homeomorphisms, i.e. after shrinking $U$ if necessary, there is a bijection $\sigma:I\to I^{\prime}$ such that the composition $\pi_{\sigma(i)}\circ f:U\to G_{\sigma(i)}^{\prime}$ factors through the projection $\pi_{i}:U\to G_{i}$. In particular, if $U=\prod_{i}U_{i}$ where each $U_{i}$ is connected then $f$ is a product.
	\end{thm}
	\begin{proof}
		We may assume without loss of generality that $I = I'$ and $\mathfrak{g}_{i} = \mathfrak{g}'_{i}$ for all $i \in I$, and so there is a measurable function $\sigma: U \to \mathrm{Perm}(I)$ such that $Df(x)(\mathfrak{g}_{i}) = \mathfrak{g}_{\sigma(x)(i)}$ for a.e. $x \in U$, since Lemma 7.7 in \cite{kleiner2021pansu}.
		
		Set 
		\begin{equation*}
			K_i:=\left\{{k\in \left\{1,\dots,n\right\}: G_k\cong G_i}\right\}
		\end{equation*} 
		We know $\sigma(i) = i$ if $|K_i| = 1$. Hence we may assume without loss of generality that $|K_i| \geq 2 \quad \forall i = 1, \dots, n$.
		
		The $N_i - 1$ form $i_Y \mathrm{vol}_{G_i}$ is closed since
		\begin{equation*}
			d(i_{X_{j,k}} \mathrm{vol}_{G_j})=\mathcal{L}_{X_{j,k}}\mathrm{vol}_{G_j}=0
		\end{equation*} 
		Let $\alpha_i = i_Y \mathrm{vol}_{G_i}$. Then the $N_i - 1$ form $\alpha_i$ is left-invariant, closed and has weight $-\nu_i + 1$.
		
		For $j \in K_i$ let $X_{j,k}$, $k = 1, \dots, \dim V_1(G_i)$ be a basis of $V_1(G_j)$. Then $i_{X_{j,k}} \mathrm{vol}_{G_j}$, $k = 1, \dots, \dim V_1(G_i)$ is a basis of the left-invariant forms on $G_j$ with degree $N_i - 1$ and weight $-\nu_i + 1$. Since pullback by a graded isomorphism preserves degree and weight we have
		
		\[
		f_P^* \alpha_i = \sum_{j \in K_i} \sum_{k = 1}^{\dim V_1(G_i)} a_{j,k} \, i_{X_{j,k}} \mathrm{vol}_{G_j}
		\]
		
		with $a_{j,k} \in L^{1}_{\mathrm{loc}}(U)$. Set
		
		\[
		a_j := \bigl( a_{j,1}, \dots, a_{j,\dim V_1(G_i)} \bigr)
		\]
		
		and
		
		\[
		E_j = \bigl\{ x \in U : \sigma^{-1}(x)(i) = j \bigr\}.
		\]
		
		Then $U \setminus \bigcup_{j \in K_i} E_j$ is a null set. Since $D_P f(x)$ is a graded automorphism for a.e. $x \in U$, we have, for all $j \in K_i$,
		
		\begin{equation}
			a_j \neq 0 \ \text{a.e. in } E_j, \qquad a_j = 0 \ \text{a.e. in } U \setminus E_j. \tag{6.11}
		\end{equation}
		
		Let \(\theta_{j',k'}\) be a basis of left-invariant one-forms which vanish on \(\oplus_{l=2}^{s}V_{l}(G_{j'})\) and which is dual to the basis \(X_{j',k'}\) of \(V_{1}(\mathfrak{g}_{j'})\), i.e. \(\theta_{j',k'}(X_{j',k})=\delta_{kk'}\). 
		Note that the forms \(\theta_{j',k'}\) are closed. 
		For \(l\in\{1,\ldots,n\}\setminus\{j'\}\), and \(X\in V_{1}(G_{l})\) consider the closed form
		\[
		\beta=\theta_{j',k'}\wedge i_X\operatorname{vol}_{G_{l}}\wedge\Bigl(\bigwedge_{i'\neq j',l}\operatorname{vol}_{G_{i'}}\Bigr).
		\]
		Then, for a.e. \(x\in U\),
		\[
		(D_{P}f)^{*}(x)\alpha_{i}\wedge\beta=\pm a_{j',k'}\,i_X\operatorname{vol}_{G}.
		\]
		In view of the assumption of $p$ we get from the Pullback Theorem \ref{Pullback2},
		\begin{equation*}
			\begin{aligned}
				0&=\int_{U}f^{*}\alpha\wedge\beta\wedge d\varphi\\
				&=\int_{U}f_{P}^{*}\alpha\wedge\beta\wedge d\varphi\\
				&=\pm\int_{U}a_{j',k'}\,X\varphi\,\operatorname{vol}_{G}\\
				&=\pm\int_{U}Xa_{j',k'}\,\varphi\,\operatorname{vol}_{G}
			\end{aligned}
		\end{equation*}
		for all \(\varphi\in C_{c}^{\infty}(U)\), because $wt(\alpha)+wt(\beta\wedge d\varphi)\leq \nu_1$. 
		Thus,$X a_{j',k'} = 0$ distributionally. Since $V_{1} \cap \mathfrak{g}_{j^{\prime}}$ generates $\mathfrak{g}_{j^{\prime}}$ as a Lie algebra, we get that $Z a_{j',k'} = 0$ for all $Z \in \mathfrak{g}_{j^{\prime}}$. Arguing as in Theorem 7.1 in \cite{kleiner2021pansu}, we get that there is a $\sigma(j) \in I_{0}$ such that $a_{j} \neq 0$ and $a_{k} = 0$ a.e. in $U$. The index $i_{0}$ was arbitrary, so we conclude that there is a permutation $\sigma \in \operatorname{Perm}(I)$ such that $Df(x)(\mathfrak{g}_{i}) = \mathfrak{g}_{\sigma(i)}$ for a.e. $x \in U$. 
		
	\end{proof}
	
	The product structure established in Theorem \ref{thm-Rigidity} simplifies the regularity analysis significantly. Since $f$ locally factors as $f = f_1 \times \dots \times f_n$ and the Pansu differential is an isomorphism, the horizontal gradient $D_H f$ inherits the integrability properties of its components. When $p$ exceeds the maximum homogeneous dimension $Q$, we can apply the Sobolev embedding theorem component-wise to deduce optimal Hölder continuity for the entire mapping, as detailed in the following theorem.

	\begin{thm}\label{Holder}
		The Sobolev mapping $f\in W^{1,p}$ mentioned in theorem \ref{thm-Rigidity} is locally $(1-\frac{Q}{p})$-Hölder continuous, where $p>Q$ and $Q=\max \left\{\text{homogeneous dim of}\ G_i\right\}$.
	\end{thm}
	\begin{proof}
		From theorem \ref{thm-Rigidity}, we can decompose $f$ as the form of product. For convenience, assume 
		\begin{equation*}
			\begin{aligned}
				&f(x_1, x_2,\dots, x_n)=f_1(x_1)f_2(x_2)\cdot \cdot \cdot f_n(x_n)\\
				&f_i:G_i\to G_i'
			\end{aligned}
		\end{equation*}
		Therefore, the Pansu differential of $f$ has the following form
		\begin{equation*}
			D_Pf=
			\begin{pmatrix}
				D_Pf_1& & &\\
				&D_Pf_2 & &\\
				&       &\ddots &\\
				&       & & D_Pf_n
			\end{pmatrix}
		\end{equation*}
		and
		\begin{equation*}
			D_Hf=
			\begin{pmatrix}
				D_Hf_1& & &\\
				&D_Hf_2 & &\\
				&       &\ddots &\\
				&       & & D_Hf_n
			\end{pmatrix}
		\end{equation*}
		It can also be written as 
		\begin{equation*}
			D_Hf_i\circ \pi_i=\pi_i \circ D_Hf_i.
		\end{equation*}
		Since $\left|D_Hf_i(x_i)\right| \leq \left|D_Hf(x)\right|$,
		for each $f_i$, we have
		\begin{equation*}
			\begin{aligned}
				&\int_{U_i}\left|D_Hf_i(x_i)\right|^q d\mu_i\\
				&=\frac{1}{\prod_{j\ne i}\mu_j(U_j)} \int_{U_1}\dots  \int_{U_{i-1}} \int_{U_{i+1}}\dots \int_{U_n} \left(\int_{U_i}\left|D_Hf_i(x_i)\right|^p d\mu_i\right) d\mu_1 \dots d\mu_{i-1} d\mu_{i+1}\dots d\mu_n \\
				&=\frac{1}{\prod_{j\ne i}\mu_j(U_j)} \int_{U} \left|D_Hf_i(\pi_{i}(x))\right|^p d\mu \\
				&=\frac{1}{\prod_{j\ne i}\mu_j(U_j)} \int_{U} \left|\pi_{i}\circ D_Hf_i(x)\right|^p d\mu \\
				&\leq\frac{1}{\prod_{j\ne i}\mu_j(U_j)} \int_{U} \left|D_Hf(x)\right|^p d\mu \\
				&< \infty
			\end{aligned}
		\end{equation*}
		Thus, locally each $f_i\in W_{loc}^{1,p}$ for any $1\leq i \leq n$. By the definition of $Q$ and $p>Q$, each $f_i\in C_{loc}^{0,1-\frac{Q}{p}}$ because of Sobolev embedding theorem. That means $f$ is locally $(1-\frac{Q}{p})$-Hölder continuous.
		\begin{equation*}
			\begin{aligned}
				d_{cc}(f(x),f(y))&= d_{cc}(f_1(x_1)f_2(x_2)\cdot \cdot \cdot f_n(x_n),f_1(y_1)f_2(y_2)\cdot \cdot \cdot f_n(y_n))\\
				&\leq d_{cc}(f_1(x_1)f_2(x_2)\cdot \cdot \cdot f_n(x_n),f_1(y_1)f_2(x_2)\cdot \cdot \cdot f_n(x_n))\\
				&+  d_{cc}(f_1(y_1)f_2(x_2)\cdot \cdot \cdot f_n(x_n),f_1(y_1)f_2(y_2)f_3(x_3)\cdot \cdot \cdot f_n(x_n))\\
				&+\dots\\
				&+ d_{cc}(f_1(y_1)f_2(y_2)\cdot \cdot \cdot f_{n-1}(y_{n-1})f_n(x_n),f_1(y_1)f_2(y_2)\cdot \cdot \cdot f_n(y_n))\\
				&=d_{cc}(f_1(x_1),f_1(y_1))+d_{cc}(f_2(x_2),f_2(y_2))+\dots +d_{cc}(f_n(x_n),f_n(y_n))\\
				&\leq C_1d_{cc}(x_1,y_1)^{1-\frac{Q}{p}}+C_2d_{cc}(x_2,y_2)^{1-\frac{Q}{p}}+\dots +C_nd_{cc}(x_n,y_n)^{1-\frac{Q}{p}}\\
				&\leq n\cdot\max \left\{C_1,\dots, C_n\right\}d(x,y)^{1-\frac{Q}{p}}
			\end{aligned} 
		\end{equation*} 
	\end{proof}
	
	$\mathbf{Example}$: We can construct a Sobolev mapping $F\in W^{1,p}$ and $F:G\to G$ with $Q-1<p<Q$	where $Q=\max \left\{\text{homogeneous dim of}\ G_i\right\}$ and $G=H_n\cdot H_n \cdot \dots \cdot H_n$ is the product of Heisenberg groups.
	
	On Heisenberg group, we have inversion
	\begin{equation*}
		\begin{aligned}
			&\mathcal{J}(x,y,t):=\left(\frac{-x}{\|(x,y,t)\|_\mathbb{H}^2},\frac{-y}{\|(x,y,t)\|_\mathbb{H}^2},\frac{-t}{\|(x,y,t)\|_\mathbb{H}^4}\right)\\
		\end{aligned}	
	\end{equation*}
	where $\|(x,y,t)\|_\mathbb{H}=\left((x^2+y^2)^2+16t^2\right)^\frac{1}{4}$.
	
	The upper gradient of it is
	\begin{equation*}
		\|D_H\mathcal{J}_p\|\approx \frac{1}{\|p\|_\mathbb{H}^2}
	\end{equation*}
	Because 
	\begin{equation*}
		\int_{B(e,1)}\left|D_H\mathcal{J}_p\right|^q dp\approx \int_{0}^{1} r^{-2q}\cdot r^{Q-1}dr, 
	\end{equation*}
	we know that $\left|D_H\mathcal{J}_p\right|\in L^q$ for $q<\frac{Q}{2}$.
	
	Therefore, in order to find a map in $W^{1,p}$ for $2n+1<p<2n+2=Q$, we are supposed to make some amendment for $\mathcal{J}$.
	At first, we denote 
	\begin{equation*}
		g_n(p)=\delta_\frac{1}{(2-\varepsilon)^n}\left(\mathcal{J}(p)\right)\cdot I_{\left\{\frac{1}{2^n}\leq\|p\|_\mathbb{H} <\frac{1}{2^{n-1}}\right\}}
	\end{equation*}
	and 
	\begin{equation*}
		\begin{aligned}
			f_1(p)&=g_1(p)\\
			f_{n+1}(p)&=
			\begin{cases}
				f_n   &,  \frac{1}{2-\varepsilon} \leq f_n <\frac{2^n}{(2-\varepsilon)^{n+1}} \\
				g_{n+1}
				&,  \frac{2^n}{(2-\varepsilon)^{n+1}}\leq g_{n+1}<\frac{2^{n+1}}{(2-\varepsilon)^{n+1}}
			\end{cases}   \\
			f(p)&=\lim_{n \to \infty} f_n(p)
		\end{aligned}
	\end{equation*}
	for $\varepsilon>0$ small enough. When $p\to e$, $\|f(p)\|_\mathbb{H}\lesssim \sum_{n=1}^{\infty}\frac{2^n}{(2-\varepsilon)^n}I_{\left\{\frac{1}{2^n}\leq\|p\|_\mathbb{H} <\frac{1}{2^{n-1}}\right\}} \to \infty$.
	Then 
	\begin{equation*}
		\begin{aligned}
			\left|D_Hf\right|(p)&\lesssim \sum_{n=1}^{\infty}\left|D_Hg_n\right|(p)\\
			&\leq \sum_{n=1}^{\infty}\frac{2^n}{(2-\varepsilon)^n}I_{\left\{\frac{1}{2^n}\leq\|p\|_\mathbb{H} <\frac{1}{2^{n-1}}\right\}}
		\end{aligned} 
	\end{equation*}
	Thus, we know
	\begin{equation*}
		\begin{aligned}
			\int_{B(e,1)}\left|D_Hf\right|^q(p)dp&\leq \sum_{n=1}^{\infty} \int_{\frac{1}{2^n}\leq\|p\|_\mathbb{H} <\frac{1}{2^{n-1}}}\left|D_Hf\right|^q(p)dp\\
			&\leq \sum_{n=1}^{\infty} \int_{\frac{1}{2^n}\leq\|p\|_\mathbb{H} <\frac{1}{2^{n-1}}}\frac{2^{nq}}{(2-\varepsilon)^{nq}}dp\\
			&\leq \sum_{n=1}^{\infty} \frac{2^{nq}}{(2-\varepsilon)^{nq}}\cdot 2^{nQ}\\
		\end{aligned}
	\end{equation*}
	In order to make $\left|D_Hf\right|\in L^q$ for $2n+1<q<2n+2=Q$,
	the $\varepsilon$ need to suffice $\varepsilon<2-\exp(-\frac{2n+2}{2n+1})$. Thus by taking this kind $\varepsilon$, we get a locally non-bounded Sobolev mapping $f\in W^{1,q}(B(e,1),H_n)$ for $2n+1<q<2n+2=Q$.
	
	Let's go back to $G=H_n\cdot H_n \cdot \dots \cdot H_n$ the product of Heisenberg groups. We now can construct a locally non-bounded Sobolev mapping $F\in W^{1,q}(B(e,1)\times B(e,1)\times \dots \times B(e,1) ,G)$ for $2n+1<q<2n+2=Q$. The $F(p_1,p_2,\dots,p_m)=f_1(p_1)\cdot\dots\cdot f_m(p_m)$ where all the $f_i$ are the $f$ mentioned above.  
	
	\begin{thm}\label{product-quo}
		Let $G = \tilde{G} / \exp(K)$ be a product quotient of homogeneous dimension $\nu$, and $\mathfrak{g} = V_1 \oplus V_2$ be the layer decomposition of its graded Lie algebra. Suppose $p > Q-1$, $U \subset G$ is open, and $f : G \supset U \to G$ is a $W^{1,p}_{\mathrm{loc}}$-mapping such that $D_P f(x) : \mathfrak{g} \to \mathfrak{g}$ exists almost everywhere and the sign of its determinant is constant almost everywhere. Then there is a locally constant assignment of a permutation $U \ni x \mapsto \sigma_x \in S_n$ such that:\\
		(1)\  For a.e. $x \in U$, the Pansu derivative $D_P f(x)$ permutes the subalgebras $\mathfrak{g}_1, \dots, \mathfrak{g}_n \subset \mathfrak{g}$ in accordance with the permutation $\sigma$:
		\[
		D_P f(x)(\mathfrak{g}_i) = \mathfrak{g}_{\sigma_x(i)}.
		\]
		
		(2)\  For every $x \in U$, cosets of $G_i$ are mapped to cosets of $G_{\sigma_x(i)}$ locally near $x$: there is neighborhood $V_x$ of $x$ such that for every $1 \leq j \leq n$, and every $y \in V$,
		\[
		f(V_x \cap yG_i) \subset f(y)G_{\sigma_x(i)}.
		\]
	\end{thm}
	\begin{proof}
		We can use the same way to prove as in \cite{kleiner2021pansu}, but instead we use the modification in theorem \ref{approx}. 
	\end{proof}
	From (2) in theorem \ref{product-quo}, locally $f(V_x \cap yG_i) \subset f(y)G_{\sigma_x(i)}$, which means we can define the lifting $F:\tilde{G}\to \tilde{G}$. For convenience, assume $f(e)=e$.
	\begin{equation}
		f(y_i)=f(e\cdot y_i)=f(e)\cdot f(y_i) \in f(e)G_{\sigma_x(i)},
	\end{equation} 
	where $y_i\in G_i$.
	Thus we can define each $f_i:=f|_{G_i}$. Consequently, $F$ can be defined as 
	\begin{equation}
		F(y_1,y_2,\dots,y_n)=f_1(y_1)\cdot\dots\cdot f_n(y_n)
	\end{equation}
	With the help of the Sobolev contact lift theorem \cite{hakavuori2025smooth}, the mapping $f:G\to G$ has the same integrability as contact lift $F:\tilde{G}\to \tilde{G}$. That is to say $f\in W^{1,p}(G,G)$ and $F\in W^{1,p}(\tilde{G},\tilde{G})$.
	\begin{thm}
		The $f$ mentioned in the theorem \ref{product-quo}, with $p>Q$, is $(1-\frac{Q}{p})$-Hölder continuous locally.  
	\end{thm} 
	\begin{proof}
		Since $D_{P}f(x)$ is an isomorphism for $\mu$-a.e. $x\in V$, the $D_{H}f(x)$ is also an isomorphism for $\mu$-a.e. $x\in V$. Because $F$ is the contact lift of $f$ and the central extension is not for the horizontal layer, $D_HF=D_Hf$ is an isomorphism for $\mu$-a.e. $x\in V$. Due to Lie algebra homorophism $D_PF$, $D_PF$ is an isomorphism for $\mu$-a.e. $x\in V$. Applying theorem \ref{Holder}, $F$ is $(1-\frac{Q}{p})$-Hölder continuous locally.
		
		With the commuting diagram
		\begin{displaymath}
			\xymatrix{
				\tilde{G} \ar[d]^{F} \ar[r]^{\pi} & G \ar[d]^{f}
				\\
				\tilde{G} \ar[r]^{\pi} & G
			}
		\end{displaymath}
		thus
		\begin{equation}
			\begin{aligned}
				d_{G}(f(a),f(b))&=d_{G}(f\circ\pi (a),f\circ\pi (b))\\
				&=d_{G}(\pi\circ F (a),\pi\circ F (b))\\
				&\leq \|\pi\|d_{\tilde{G}}(F (a),F (b))\\
				&\leq C\cdot \|\pi\|d_{\tilde{G}}(a,b)^{1-\frac{Q}{p}}\\
				&\leq C\cdot \|\pi\|d_{G}(a,b)^{1-\frac{Q}{p}}
			\end{aligned}
		\end{equation}
		the last inequality is because the geodesic $\gamma_{ab}$ connecting $a$ and $b$ on $G$ is in the set of $\left\{\text{the horizontal curves on $\tilde{G}$ connecting $a$ and $b$}\right\}$, thus $\gamma_{ab}$ is not the shortest horizontal curve connecting $a$ and $b$, which means $d_{\tilde{G}}(a,b) \leq\left|\gamma_{ab} \right|$. 
		
		It is easy to show that there is a constant $C_1$ such that the $\|\pi\|\leq C_1$. Because the coordinate of $F(a)F(b)^{-1}=\left(\pi \circ F(a) (\pi \circ F(a))^{-1},\lambda(a,b)\right)$ where the $\lambda(a,b)$ is the mapping from $G\times G\to \exp(K)$
		\begin{equation}
			\begin{aligned}
				d_{\tilde{G}}(F(a)F(b)^{-1},0)&\approx \|\pi \circ F(a) (\pi \circ F(a))^{-1}\|_\mathbb{K}+\|\lambda(a,b)\|_\mathbb{R}^{\frac{1}{2}}\\
				&\geq \|\pi \circ F(a) (\pi \circ F(a))^{-1}\|_\mathbb{K}\\
				&\approx d_{G}(\pi \circ F(a) (\pi \circ F(a))^{-1},0)\\
				&=d_{G}(\pi\circ F (a),\pi\circ F (b)),
			\end{aligned}
		\end{equation}
		where the $\|\cdot\|_\mathbb{K}$ is the quasimetric mentioned in \eqref{metric-K} and the $\|\cdot\|_\mathbb{R}$ is the Euclidean metric. 
		
		In conclusion, $d_{G}(f(a),f(b))\lesssim d_{G}(a,b)^{1-\frac{Q}{p}}$ locally.
	\end{proof}
	
	\section{The sharpness of condition $wt(\omega)+wt(d\eta)\leq \nu_1$}\label{sec:sharpness}
	Theorems \ref{Pullback} and \ref{thm-Pull-0} relied crucially on the weight condition $wt(\omega) + wt(d\eta) \le \nu_1$. This condition is not merely technical; it reflects the scaling limit of the Carnot group geometry relative to the Sobolev exponent. To demonstrate that this condition is sharp---meaning the conclusions fail if it is violated---we construct explicit counterexamples using the Heisenberg inversion map. These examples show that the higher-layer components of the differential cannot be controlled if the form weights are too large.
	
	\begin{thm}
		The condition $wt(\omega)+wt(d\eta)\leq \nu_1$ in theorem \ref{thm-Pull-0} is sharp.
	\end{thm}
	\begin{proof}
		We would give a proof by contradiction. If the condition is removable, the we can get the following claim.
		
		Claim: For $f$ mentioned in theorem \ref{thm-Pull-0}, $Z\left(det(D_Hf)\right)=0$, where $Z$ is a left-invariant vector field in the $i$-th layer for $i\geq 2$.
		
		Let's start with the case on Heisenberg group $\mathbb{H}^1$. 
		
		Since $f$ is Pansu differential a.e., $D_f$ is a Lie algebra homomorphism a.e.. $f_P^*(dx_1\wedge dx_2)=adx_1\wedge dx_2$. By theorem \ref{thm-Pull-0},
		\begin{equation*}
			\begin{aligned}
				0&=\int_{U} f_P^*(dx_1\wedge dx_2)\wedge d\varphi\\
				&=\int_{U}adx_1\wedge dx_2 \wedge(X_1\varphi dx_1 +X_2\varphi dx_2 +Z \varphi \alpha)\\
				&=\int_{U}aZ \varphi dV\\
				&=-\int_{U}Za \varphi dV,
			\end{aligned}
		\end{equation*}
		where $\varphi$ is any smooth test function. Then $Za=0$ a.e., where $a(p)=det(D_Hf)(p)$.
		
		For general step-$n$ Carnot group $G$, denote the left-invariant co-vector field $dx_1, dx_2, \dots ,dx_k$ as the basis of the first layer and $\left\{\omega_j\right\}$ are the left-invariant co-vector field in higher layers. We have $f_P^*(dx_1\wedge dx_2\wedge \dots \wedge dx_k)=adx_1\wedge dx_2\wedge \dots \wedge dx_k$.
		
		Because $dx_1\wedge dx_2\wedge \dots \wedge dx_k$ is closed, 
		\begin{equation*}
			\begin{aligned}
				0&=\int_{U} f_P^*(dx_1\wedge dx_2\wedge \dots \wedge dx_k)\wedge d\left(\varphi\omega_{1}\wedge \omega_{2}\wedge \dots \wedge\omega_{j-1}\wedge\omega_{j+1}\wedge\dots \wedge \omega_m\right)\\
				&=\int_{U} f_P^*(dx_1\wedge dx_2\wedge \dots \wedge dx_k)\wedge \varphi d\left(\omega_{1}\wedge \omega_{2}\wedge \dots \wedge\omega_{j-1}\wedge\omega_{j+1}\wedge\dots \wedge \omega_m\right)\\
				&\pm \int_{U} f_P^*(dx_1\wedge dx_2\wedge \dots \wedge dx_k)\wedge \left(\omega_{1}\wedge \omega_{2}\wedge \dots \wedge\omega_{j-1}\wedge\omega_{j+1}\wedge\dots \wedge \omega_m\right)\wedge d \varphi\\
				&=\pm \int_{U} adx_1\wedge dx_2\wedge \dots \wedge dx_k \wedge \left(\omega_{1}\wedge \omega_{2}\wedge \dots \wedge\omega_{j-1}\wedge\omega_{j+1}\wedge\dots \wedge \omega_m\right)\wedge d \varphi\\
				&=\pm \int_{U} a\left(\iota_{Z_j}dV\right)\wedge d\varphi\\
				&=\pm \int_{U} a\left(\iota_{Z_j}dV\right)\wedge (\sum X_i\varphi dx_i+\sum Z_l\varphi \omega_l)\\
				&=\pm \int_{U} aZ_j\varphi dV\\
				&=\mp \int_{U} Z_ja\varphi dV,
			\end{aligned}
		\end{equation*}
		where the third equation is because $d\omega=dx\wedge (\cdot)$. That means $Z_jdet(D_Hf)=Z_ja=0$ a.e. for any vector field in the $i$-th layer $i\geq 2$.
		
		However, we know that the inversion on Heisenberg group 
		\begin{equation*}
			\begin{aligned}
				&\mathcal{J}(x,y,t):=\left(\frac{-x}{\|(x,y,t)\|_\mathbb{H}^2},\frac{-y}{\|(x,y,t)\|_\mathbb{H}^2},\frac{-t}{\|(x,y,t)\|_\mathbb{H}^4}\right)\\
			\end{aligned}	
		\end{equation*}
		is a locally smooth and Pansu differentiable mapping expect for origin point. The upper gradient of it is
		\begin{equation*}
			\|D_H\mathcal{J}_p\|\approx \frac{1}{\|p\|_\mathbb{H}^2}
		\end{equation*}
		which is not invariant under the movement of $Z$. A contradiction. 
	\end{proof}

	\section{Non-embedding theorem for contact Sobolev mappings}\label{sec:nonembedding}
	The stratified convergence results from Theorem \ref{thm-approx} and \ref{thm-approx-higher} have profound geometric consequences beyond rigidity. In particular, they impose restrictions on the existence of embeddings between Carnot groups of different ranks. By combining the mollification technique with the Gromov non-embedding strategy, we can relax the classical continuity requirements. 
	
	For Gromov non-embedding theorem, {Haj{\l}asz} gives a simple proof \cite{hajlasz2025holder}. However, this result requires $f$ is a $\gamma$-Hölder continuous embedding for $\gamma>\frac{1}{2}$ at least. Therefore, we want to know if we can lower the requirements for $\gamma$ through Sobolev property. An easy situation is for contact Sobolev mappings $W^{1,p}(\mathbb{R}^N;G_2)$, since the source group is just an Euclidean space.
	
	\begin{thm}
		Suppose that $G_2$ is a step-$n$ Carnot group. The Lie algebra of $G_2$ is $g_2=g_2^{[1]} \oplus g_2^{[2]} \oplus \dots \oplus g_2^{[n]}$ satisfying $N>dim(g_2^{[1]})$. Then there does not exist a topological embedding $f\in C^{0} \cap W^{1,p}(\Omega;G_2)$ for $p\geq max\left\{N,n+1\right\}$ and $\Omega \subset \mathbb{R}^N$ is a domain. 
	\end{thm}
	\begin{proof}
		The easiest case is for $p>N$. Because of Sobolev embedding theorem on Carnot group, $f$ is a Hölder continuous mapping and is Pansu differentiable almost everywhere. This would give a strong restriction by Lie algebra homomorphism.
		
		Since $N>dim(g_2^{[1]})$, $D_Pf$ is not full-rank which means any $f\in C^0 \cap W^{1,p}(G_1;G_2)$ is not an embedding. 
		
		For $n+1<p\leq N$, $f$ is not certainly Pansu differentiable almost everywhere. We can not rely on strong restriction of Lie algebra homomorphism to prove, but as a Sobolev contact mapping $f$ would still preserve stratified structure by theorem \ref{thm-approx} and \ref{thm-approx-higher}.
		
		It suffices to show that there does not exist an topological embedding $f\in C^0 \cap W^{1,p}(\mathbb{B}^{N};G_2)$. We will prove by contradiction.
		
		Suppose that $f$ is an embedding, then $f|_{\mathbb{S}^{N-1}}$ is also a continuous embedding. From theorem 8.10 in \cite{hajlasz2025holder}, there is $\omega \in \Omega_c^{N-1}(G_2)$ such that $d\omega \in \Omega_c^{N}(G_2)$ and 
		\begin{equation*}
			\int_{\mathbb{S}^{N-1}}f^*\omega =1.
		\end{equation*}
		
		However, by approximation of mollification and stokes formula
		\begin{equation*}
			\int_{\mathbb{S}^{N-1}}f^*\omega =\int_{\mathbb{S}^{N-1}}f_\varepsilon^*\omega= \int_{\mathbb{B}^{N}}df_\varepsilon^*\omega= \int_{\mathbb{B}^{N}}f_\varepsilon^*(d\omega)
		\end{equation*} 
		By theorem \ref{thm-approx} and \ref{thm-approx-higher} and $N>dim(g_2^{[1]})$, $f_\varepsilon^*(d\omega)$ will inevitably vanish in $\left(N-dim(g_2^{[1]})\right)$ directions while the other directions would not expand.
		\begin{equation*}
			\int_{\mathbb{B}^{N}}f_\varepsilon^*(d\omega) \to 0 
		\end{equation*}
		A contradiction.
		
	\end{proof}
	
	If we drop the assumption that $f$ is continuous, the $N-1$ dimensional Hausdorff measure of the image of hypersurface is zero.
	\begin{thm}
		Suppose that $G_2$ is a step-$n$ Carnot group. The Lie algebra of $G_2$ is $g_2=g_2^{[1]} \oplus g_2^{[2]} \oplus \dots \oplus g_2^{[n]}$ satisfying $N>dim(g_2^{[1]})$. Suppose that $f\in  W^{1,p}(\mathbb{B}^N;G_2)$ for $p\geq max\left\{N,n+1\right\} $ and $\mathbb{B}^N \subset \mathbb{R}^N$ is a domain. Then $\mathcal{H}^{N-1}\left(f(\mathbb{S}^{N-1}(r))\right)=0$ for $N-1$ dimensional sphere with radius $r$ for $r>0$.
	\end{thm}
	\begin{proof}
		For convenience, we can assume the center of $\mathbb{S}^{N-1}(r)$ is original point. Taking spherical polar coordinates, the volume form on $\mathbb{B}^N$ is 
		\begin{equation*}
			d\sigma_r \wedge dr=r^{N-1}d\sigma_1 \wedge dr,
		\end{equation*} 
		where the $d\sigma_1$ is the area measure on $\mathbb{S}(1)$.
		with it,
		\begin{equation}
			\begin{aligned}
				\int_{B(e,1)} f^*(d\sigma_1) \wedge r^{N-1}dr&=
				\frac{1}{N}\int_{B(e,1)} f^*(d\sigma_1) \wedge d(r^N)\\
				&= \frac{1}{N}\int_{B(e,1)} df^*(d\sigma_1) \wedge r^N\\
				&= \frac{1}{N}\int_{B(e,1)} r^Nf^*(dV)  \\
				&\leq \frac{1}{N}\sum_{i=0}^{\infty}\int_{B(\frac{1}{2^i})\backslash B(\frac{1}{2^{i+1}})} r^Nf^*(dV)  \\
				&\leq \frac{1}{N}\sum_{i=0}^{\infty}\frac{1}{2^{iN}}\int_{B(\frac{1}{2^i})\backslash B(\frac{1}{2^{i+1}})} f^*(dV)  \\
				&\leq \frac{1}{N}\sum_{i=0}^{\infty}\frac{1}{2^{iN}}\int_{B(1)}f^*(dV)  \\
				&= \frac{1}{N(1-2^{-N})}\int_{B(1)}f^*(dV)  \\
				&=0.
			\end{aligned}
		\end{equation}
		Then $\mathcal{H}^{N-1}\left(f(\mathbb{S}^{N-1}(1))\right)=0$.
		With the same method, we can prove $\mathcal{H}^{N-1}\left(f(\mathbb{S}^{N-1}(r))\right)=0$ for any $r>0$. 
	\end{proof}
	
	Not only the Sobolev contact mappings preserve stratified structures, there is a more clear description for Hölder continuous mappings. 
	We can estimate the Hausdorff dimension of the image of $\Omega \subset \mathbb{R}^m$ under mappings in $C^{0, \gamma}$.
	For $f \in C^{0, \gamma}\left(\Omega ; G\right)$, it is easy to know $\mathcal{H}_{cc}^\frac{m}{\gamma}(f(\Omega))\le C\cdot \mathcal{H}_E^m(\Omega) $, where the $\mathcal{H}_E^m$ is $m$ dimensional Euclidean Hausdorff measure. And also, $\mathcal{H}_{E}^\frac{m}{\gamma}(f(\Omega))\le C\cdot \mathcal{H}_E^m(\Omega) $ from theorem 1.1 in \cite{balogh2003comparison}.  
	We show the case of step-3 Carnot group.   
	\begin{lem}\label{lem-rate}
		Let  $\gamma:[0,1] \rightarrow \mathbb{R}^{r+s+t} $ be a curve and $ \alpha \in \left( \frac{1}{2}, 1 \right] $ .Then $ \gamma \in C^{0, \alpha}\left([0,1] ; G \right) $ if and only if $ \pi \circ \gamma \in C^{0, \alpha}\left( [0,1] ; \mathbb{R}^{r}\right) $ ,and
		\begin{equation*}
			\begin{array}{l}
				\gamma^{B_k}(b)-\gamma^{B_k}(a)=\frac{1}{2}\sum_{i=1}^{r} \sum_{i<j} \alpha_{k}^{i j}\int_{a}^{b}\left( \gamma^A_{i}d \gamma^A_{j}- \gamma^A_{j}d \gamma^A_{i}\right)\\
				\gamma^{C_{m}}(b)-\gamma^{C_{m}}(a)=-\frac{1}{2}\sum_{i=1}^{r} \sum_{j=1}^{s} \int_{a}^{b}\gamma^B_{j} \beta_{m}^{i j}d \gamma^A_{i}+\frac{1}{12} \sum_{l=1}^{r} \sum_{k=1}^{s}\sum_{i<j}  \alpha_{k}^{i j}\beta_{m}^{l k}\gamma^A_{l}\int_{a}^{b}\left( \gamma^A_{i}d\gamma^ A_{j}- \gamma^A_{j}d \gamma^A_{i}\right) 
			\end{array} 
		\end{equation*}
		If in addition $\gamma(a)=0$ , then
		\begin{equation*}
			\begin{array}{l}
				\left|\gamma^{A_k}(b)-\gamma^{A_k}(a)\right| \lesssim \varepsilon^{ \gamma}\\
				\left|\gamma^{B_k}(b)-\gamma^{B_k}(a)\right| \lesssim \varepsilon^{2 \gamma}\\
				\left|\gamma^{C_{m}}(b)-\gamma^{C_{m}}(a)\right| \lesssim \varepsilon^{3 \gamma} 
			\end{array} 
		\end{equation*}
	\end{lem}
	\begin{proof}
		The proof here is the same as Theorem 7.6 in \cite{hajlasz2025holder}. 
	\end{proof}

	\begin{thm}
		Suppose that  $f \in C^{0, \gamma}\left(\Omega ; G\right)$ , where $G=\mathbb{R}^{r+s+t}$ is a step three Carnot group   and $\Omega \subset \mathbb{R}^{m}$  is open.\\ 
		(1) If $ 0 <\gamma \leq \frac{1}{2}$ and $\frac{m+2r\gamma+s\gamma}{3\gamma}\le r+2s+3t$, then the Euclidean Hausdorff measure satisfy $\mathcal{H}_E^\frac{m+2r\gamma+s\gamma}{3\gamma}(f(\Omega))\le C\cdot \mathcal{H}_E^m(\Omega) $.\\
		(2) If $ \frac{1}{2}<\gamma \leq 1$ and $m>r\gamma$, then $dim_E(f(\Omega))\leq \frac{r}{\gamma}$.\\
		(3) If $ \frac{1}{2}<\gamma \leq 1$ and $m\leq r\gamma$, then $\mathcal{H}_E^\frac{r}{\gamma}(f(\Omega))\le C\cdot \mathcal{H}_E^m(\Omega) $.
		
		\begin{proof}
			Let us start with the proof of (1). This is a very easy case and it holds for the cases (2) and (3) as well.
			We take a covering of $\Omega$, $\Omega\subset \bigcup_{i=1}^{n}B(x_i,r_i)$ and $max(r_i)\le\delta$, then we have 
			\begin{equation*}
				\mathcal{H}_{\delta,E}^m(\Omega)\le \sum_{i=1}^{n}(r_i)^m 
			\end{equation*}
			For all the points in $G=\mathbb{R}^{r+s+t}$, we just need to study the behavior near the origin. 
			For the other points we can use group translation to move them to the origin.
			So because of Lemma \ref{lem-rate}, we have
			\begin{equation*}
				\begin{array}{l}
					\left|f^{A_k}(b)-f^{A_k}(a)\right| \lesssim \left|b-a\right|^{ \gamma}\\
					\left|f^{B_k}(b)-f^{B_k}(a)\right| \lesssim \left|b-a\right|^{2 \gamma}\\
					\left|f^{C_{m}}(b)-f^{C_{m}}(a)\right| \lesssim \left|b-a\right|^{3 \gamma} 
				\end{array} 
			\end{equation*}
			We can use many balls respect to the CC metric, to cover the image $f(B(0,r_i))$.
			By \eqref{metric-K},
			\begin{equation*}
				\begin{aligned}
					|\pi_{i}^j(f(q)^{-1}f(p))|^{\frac{1}{i}}&\leq d_K(f(p),f(q))\\
					&\approx d_{c c}(f(p),f(q))\\
					&\lesssim |p-q|^\gamma.
				\end{aligned}
			\end{equation*}
			For convenience, we can assume that $f(q)=e$. Then
			\begin{equation}\label{6.3}
				|f^{A_l}(p)-f^{A_l}(q)|\lesssim |p-q|^\gamma
			\end{equation}
			\begin{equation*}
				|f^{B_l}(p)-f^{B_l}(q)|\lesssim |p-q|^{2\gamma}
			\end{equation*}
			\begin{equation*}
				|f^{C_l}(p)-f^{C_l}(q)|\lesssim |p-q|^{3\gamma}
			\end{equation*}
			
			In the directions of $\mathbb{R}^r$ , we can use approximately $\left[\frac{r^{\gamma}}{r^{3\gamma}}\right]$ balls with radius $r^{3\gamma}$ to cover. In the directions of $\mathbb{R}^s$ , we can use approximately $\left[\frac{r^{2\gamma}}{r^{3\gamma}}\right]$ balls with radius $r^{3\gamma}$ to cover. In the directions of $\mathbb{R}^t$ , we can use approximately $\left[\frac{r^{3\gamma}}{r^{3\gamma}}\right]$ balls with radius $r^{3\gamma}$ to cover. 
			\begin{equation}
				\begin{aligned}
					\mathcal{H}_E^k(f(\Omega))&\le \sum_{i=1}^{n} \sum^{N_i} (r_i^{3\gamma})^k\\
					& \lesssim \sum_{i=1}^{n} \left[\frac{r_i^{\gamma}}{r_i^{3\gamma}}\right]^r \left[\frac{r_i^{2\gamma}}{r_i^{3\gamma}}\right]^s \left[\frac{r_i^{3\gamma}}{r_i^{3\gamma}}\right]^t r_i^{3\gamma k}\\
					& \le \sum_{i=1}^{n} r_i^{\gamma (3k-2r-s)}
				\end{aligned}
			\end{equation}
			If we take $\gamma (3k-2r-s)=m$ , then 
			\begin{equation}
				\begin{aligned}
					\mathcal{H}_E^\frac{m+2r\gamma+s\gamma}{3\gamma}(f(\Omega))
					& \lesssim \sum_{i=1}^{n} r_i^{m}\rightarrow \mathcal{H}_E^m(\Omega)
				\end{aligned}
			\end{equation}
			
			For case (2), $ \frac{1}{2}<\gamma \leq 1$ and $m>r\gamma$.
			The first layer of $G=\mathbb{R}^{r+s+t}$ is full and the only direction the image of $f$ can move is in the second and third layer. However, since  $\frac{1}{2}<\gamma \leq 1$, $Xf^*(\omega)=0$ where the $\omega$ is any contact form on $G$, which shows the image of $f$ can not expand in the second and third layer.
			Thus, $\mathcal{H}_E^{\frac{r}{\gamma}+\delta}(f(\Omega))=0$ for any $\delta>0$.
			
			For case (3), $ \frac{1}{2}<\gamma \leq 1$ and $m\leq r\gamma$. Since  $\frac{1}{2}<\gamma \leq 1$, $Xf^*(\omega)=0$. Therefore, by lemma \ref{lem-rate}, we can compute the $f^{B}\circ \lambda$ and $f^C \circ \lambda$ as long as the coordinates of $f^{A}\circ \lambda$ are clear, where $\lambda$ is a smooth curve in $\Omega$. We can lift a curve in $\mathbb{R}^r$ to $G=\mathbb{R}^{r+s+t}$. Thus, by projection $\pi_1$, all information of $f\circ \lambda$ is in curve $\pi_1(f\circ \lambda)=f^A\circ \lambda$. 
			
			Let us take Euclidean balls $B(f(p_i),r_k)$ in $G$, then we can use balls $\pi \left(B(f(p_i),r_k) \right)$ with the number of $\frac{2Length(\lambda)^\gamma}{ r_k}$ to cover $f^A\circ \lambda$ for $Length(\lambda)$ and $r_k$ small enough. Therefore, by lifting, the balls $\left(B(f(p_i),r_k) \right)$ can also cover curve $f\circ \lambda$, since \eqref{6.3} and $r_k^{2\gamma}<r_k$ for $\gamma>\frac{1}{2}$.
			In conclusion, as long as we can cover the image of $f$ with Euclidean balls in the directions of the first layer. Then all of these balls naturally would cover the image of $f$.
			
			We take a covering of $\Omega$, $\Omega\subset \bigcup_{i=1}^{n}Box(x_i,r_i)$ and $max(r_i)\le\delta$, then we have 
			\begin{equation*}
				\mathcal{H}_{\delta,E}^m(\Omega)\le \sum_{i=1}^{n}(r_i)^m 
			\end{equation*}
			and the $n$ at least larger than $\frac{|\Omega|}{\delta^m}$.
			
			In the directions of $\mathbb{R}^r$ , we can use approximately $\left[\frac{r^{\gamma}}{r^{3\gamma}}\right]$ balls with radius $r^{3\gamma}$ to cover. In the directions of $\mathbb{R}^s$ , we can use approximately $\left[\frac{r^{2\gamma}}{r^{3\gamma}}\right]$ balls with radius $r^{3\gamma}$ to cover. In the directions of $\mathbb{R}^t$ , we can use approximately $\left[\frac{r^{3\gamma}}{r^{3\gamma}}\right]$ balls with radius $r^{3\gamma}$ to cover. 
			\begin{equation}
				\begin{aligned}
					\mathcal{H}_E^\frac{m}{\gamma}(f(\Omega))&\le \sum_{i=1}^{n}\left\{diam_E(f(B(x_i,r_i)))\right\}^\frac{m}{\gamma}\\
					&\lesssim \sum_{i=1}^{n}\left(C r_i^\gamma\right)^\frac{m}{\gamma}\\
					&=C^\frac{m}{\gamma} \sum_{i=1}^{n} r_i^m \to C^\frac{m}{\gamma} \mathcal{H}_E^m(\Omega)\\
				\end{aligned}
			\end{equation}
			
		\end{proof}
	\end{thm}
	If the Carnot group $G$ is replaced by step-$n$ Carnot groups for $n>3$, we can also draw a similar conclusion.

	\bibliographystyle{plain}
	\bibliography{ref}

@article{mitchell1985carnot,
	author  = {Mitchell, John},
	title   = {On {C}arnot--{C}arath\'eodory metrics},
	journal = {J. Differential Geom.},
	volume  = {21},
	number  = {1},
	pages   = {35--45},
	year    = {1985},
}

@article{pansu1989,
	author  = {Pansu, Pierre},
	title   = {M\'etriques de {C}arnot--{C}arath\'eodory et quasiisom\'etries des espaces sym\'etriques de rang un},
	journal = {Ann. of Math. (2)},
	volume  = {129},
	number  = {1},
	pages   = {1--60},
	year    = {1989},
}

@incollection{gromov1996,
	author    = {Gromov, Mikhael},
	title     = {Carnot--{C}arath\'eodory spaces seen from within},
	booktitle = {Sub-{R}iemannian Geometry},
	series    = {Progr. Math.},
	volume    = {144},
	pages     = {79--323},
	publisher = {Birkh\"auser},
	address   = {Basel},
	year      = {1996},
}

@book{capogna2007introduction,
	author    = {Capogna, Luca and Danielli, Donatella and Pauls, Scott D. and Tyson, Jeremy T.},
	title     = {An Introduction to the {H}eisenberg Group and the Sub-{R}iemannian Isoperimetric Problem},
	series    = {Progr. Math.},
	volume    = {259},
	publisher = {Birkh\"auser},
	address   = {Boston, MA},
	year      = {2007},
}

@article{hajlasz2000sobolev,
	author  = {Haj{\l}asz, Piotr and Koskela, Pekka},
	title   = {Sobolev met {P}oincar\'e},
	journal = {Mem. Amer. Math. Soc.},
	volume  = {145},
	number  = {688},
	pages   = {x+101},
	year    = {2000},
	note    = {MR1683160},
}

@book{reshetnyak1989,
	author    = {Reshetnyak, Yu. G.},
	title     = {Space Mappings with Bounded Distortion},
	series    = {Translations of Mathematical Monographs},
	volume    = {73},
	publisher = {American Mathematical Society},
	address   = {Providence, RI},
	year      = {1989},
}

@article{koranyi1985,
	author  = {Kor\'anyi, Adam and Reimann, Hans Martin},
	title   = {Quasiconformal mappings on the {H}eisenberg group},
	journal = {Invent. Math.},
	volume  = {80},
	number  = {2},
	pages   = {309--338},
	year    = {1985},
}

@article{koranyi1995,
	author  = {Kor\'anyi, Adam and Reimann, Hans Martin},
	title   = {Foundations for the theory of quasiconformal mappings on the {H}eisenberg group},
	journal = {Adv. Math.},
	volume  = {111},
	number  = {1},
	pages   = {1--87},
	year    = {1995},
}

@article{heinonen1997quasiregular,
	author  = {Heinonen, Juha and Holopainen, Ilkka},
	title   = {Quasiregular maps on {C}arnot groups},
	journal = {J. Geom. Anal.},
	volume  = {7},
	number  = {1},
	pages   = {109--148},
	year    = {1997},
}

@article{capogna1999regularity,
	author  = {Capogna, Luca},
	title   = {Regularity for quasilinear equations and $1$-quasiconformal maps in {C}arnot groups},
	journal = {Math. Ann.},
	volume  = {313},
	number  = {2},
	pages   = {263--295},
	year    = {1999},
}

@article{xie2017rigidity,
	author  = {Xie, Xiangdong},
	title   = {Rigidity of quasiconformal maps on {C}arnot groups},
	journal = {Math. Proc. Cambridge Philos. Soc.},
	volume  = {162},
	number  = {1},
	pages   = {131--150},
	year    = {2017},
	note    = {MR3581902},
}

@article{kleiner2021pansu,
	author  = {Kleiner, Bruce and M\"uller, Stefan and Xie, Xiangdong},
	title   = {Pansu pullback and rigidity of mappings between {C}arnot groups},
	journal = {arXiv e-prints},
	pages   = {arXiv:2004.09271},
	year    = {2021},
	note    = {Preprint, arXiv:2004.09271},
}

@article{kleiner2021nonrigid,
	author  = {Kleiner, Bruce and M\"uller, Stefan and Xie, Xiangdong},
	title   = {Sobolev mappings between nonrigid {C}arnot groups},
	journal = {arXiv e-prints},
	pages   = {arXiv:2112.01866},
	year    = {2021},
	note    = {Preprint, arXiv:2112.01866},
}

@article{kleiner2021rumin,
	author  = {Kleiner, Bruce and M\"uller, Stefan and Xie, Xiangdong},
	title   = {Sobolev mappings and the {R}umin complex},
	journal = {arXiv e-prints},
	pages   = {arXiv:2101.04528},
	year    = {2021},
	note    = {Preprint, arXiv:2101.04528},
}

@article{hajlasz2025holder,
	author  = {Haj{\l}asz, Piotr and Mirra, Jacob and Schikorra, Armin},
	title   = {H\"older continuous mappings, differential forms and the {H}eisenberg groups},
	journal = {arXiv e-prints},
	pages   = {arXiv:2503.11506},
	year    = {2025},
	note    = {Preprint, arXiv:2503.11506},
}

@article{hakavuori2025smooth,
	author  = {Hakavuori, Eero and Heikkil\"a, Susanna and Ikonen, Toni},
	title   = {Smooth contact lifts to central extensions of {C}arnot groups},
	journal = {arXiv e-prints},
	pages   = {arXiv:2508.14647},
	year    = {2025},
	note    = {Preprint, arXiv:2508.14647},
}

@article{balogh2003comparison,
	author  = {Balogh, Zolt\'an M. and Rickly, Matthieu and Serra Cassano, Francesco},
	title   = {Comparison of {H}ausdorff measures with respect to the {E}uclidean and the {H}eisenberg metric},
	journal = {Publ. Mat.},
	volume  = {47},
	number  = {1},
	pages   = {237--259},
	year    = {2003},
}
\end{document}